\documentclass [leqno, 10pt]{amsart}

\usepackage[all]{xy}
\usepackage{amssymb, amsmath, amsfonts, amsthm, amsbsy, amscd, amsxtra}
\usepackage[bbgreekl]{mathbbol} 
\usepackage[mathscr]{euscript}

\usepackage{url}
\usepackage{float}

\usepackage[linktocpage]{hyperref}

\usepackage{cite}

\newtheorem{theorem}{Theorem}[section]

\newtheorem{corollary}{Corollary}[section]
\newtheorem{lemma}{Lemma}[section]

\theoremstyle{definition}

\numberwithin{equation}{section}

\newcommand{\annulus}{\mathbb{A}}
\newcommand{\disc}{\mathbb{D}}
\newcommand{\tkk}{\tilde{k}(t)}
\newcommand{\kk}{k(t)}
\newcommand{\varep}{\varepsilon}
\newcommand{\hh}{h(t)}
\newcommand{\hhpm}{h_{\pm}(t)}
\newcommand{\hhp}{h_{+}(t)}
\newcommand{\hhm}{h_{-}(t)}
\newcommand{\mom}{\mu}

\newcommand{\sech}{\operatorname{sech}}
\newcommand{\csch}{\operatorname{csch}}

\newcommand{\sphere}{\mathbb{S}^{2}}

\newcommand{\dlef}{\mathsf{\Lambda}}
\newcommand{\Om}{\Omega}
\newcommand{\lne}{\mathcal{E}}

\newcommand{\imt}{\iota}

\newcommand{\cano}{\mathscr{K}}

\newcommand{\om}{\omega}

\newcommand{\sth}{{\vphantom{}}^{\ast}\!{h}}

\newcommand{\dad}{d^{\ast}}

\newcommand{\vol}{\mathsf{vol}}

\newcommand{\delbar}{\bar{\del}}

\newcommand{\sR}{\mathscr{R}}

\newcommand{\Ga}{\Gamma}

\newcommand{\lap}{\Delta}

\renewcommand{\j}{\mathsf{i}}
\renewcommand{\sp}{\sigma}

\newcommand{\la}{\lambda}
\newcommand{\ep}{\epsilon}

\newcommand{\cinf}{C^{\infty}}

\newcommand{\si}{\sigma}
\newcommand{\pr}{\partial}
\newcommand{\fd}{\mathscr{F}}

\newcommand{\sign}{\operatorname{sgn}}

\newcommand{\integer}{\mathbb{Z}}

\newcommand{\lie}{\mathfrak{L}}

\newcommand{\re}{\text{Re\,}}

\newcommand{\al}{\alpha}
\newcommand{\be}{\beta}
\newcommand{\ga}{\gamma}

\newcommand{\proj}{\mathbb{P}}

\newcommand{\del}{\partial}
\newcommand{\tensor}{\otimes}

\newcommand{\rea}{\mathbb R}
\newcommand{\com}{\mathbb C}

\makeindex

\begin{document}
\title[Ricci flows and vortex-like equations]{Ricci flows on surfaces related to the Einstein Weyl and Abelian vortex equations}
\author{Daniel J. F. Fox} 
\address{Departamento de Matem\'atica Aplicada\\ EUIT Industrial\\ Universidad Polit\'ecnica de Madrid\\Ronda de Valencia 3\\ 28012 Madrid Espa\~na}
\email{daniel.fox@upm.es}


\begin{abstract}
There are described equations for a pair comprising a Riemannian metric and a Killing field on a surface that contain as special cases the Einstein Weyl equations (in the sense of D. Calderbank) and a real version of a special case of the Abelian vortex equations, and it is shown that the property that a metric solve these equations is preserved by the Ricci flow. The equations are solved explicitly, and among the metrics obtained are all steady gradient Ricci solitons (e.g. the the cigar soliton) and the sausage metric; there are found other examples of eternal, ancient, and immortal Ricci flows, as well as some Ricci flows with conical singularities.
\end{abstract}
\maketitle

\section{Introduction}
Given a surface $M$, consider, for a pair $(h, Y)$ comprising a Riemannian metric $h$ with scalar curvature $\sR_{h}$ and a vector field $Y$ on $M$, and a fixed parameter $\varep = \pm 1$, the \textit{real vortex} equations:
\begin{align}\label{vlike}
&0 = d(\sR_{h} + 4\varep|Y|_{h}^{2}), & &\lie_{Y}h = 0.
\end{align}
The second condition of \eqref{vlike} says simply that $Y$ is a Killing field, while the first equation says that there is a constant $\tau$ such that
\begin{align}\label{realvortex}
\tau = \sR_{h} + 4\varep|Y|_{h}^{2}.
\end{align}
This parameter $\tau$ will be called the \textit{vortex} parameter of the solution $(h, Y)$. (The factor $4$ in \eqref{vlike} has no intrinsic significance, as it could be absorbed into $Y$, and has been chosen for consistency with the conventions of \cite{Fox-2dahs}). A solution $(h, Y)$ of \eqref{vlike} will be said to be \textit{trivial} if $Y$ is identically zero. In this case \eqref{vlike} forces $h$ to have constant curvature. (The example of a parallel vector field on a flat torus shows that the constancy of the curvature of a solution to \eqref{vlike} need not imply the triviality of the solution).

In section \ref{weylsection}, it is explained how to construct from a solution of \eqref{vlike} with $\varep = -1$ a solution of the Einstein Weyl equations as formulated for surfaces by D. Calderbank in \cite{Calderbank-mobius, Calderbank-twod}. In section \ref{vortexsection}, it is explained that a solution of \eqref{vlike} with $\varep = 1$ gives rise to a solution of the usual Abelian vortex equations on the bundle of holomorphic one forms, and can be seen as the real part of such a vortex solution. Because of these observations, a solution $(h, Y)$ to the equations \eqref{vlike} will be called \textit{Einstein Weyl} or \textit{vortex-like} as $\varep = -1$ or $\varep = 1$. As is also explained in section \ref{vortexsection} the Einstein Weyl case admits a complex reformulation in which it resembles the Abelian vortex equations but with a sign change on one term (see \eqref{signedvortexeqns}). Such \textit{signed vortex equations} and their relation with Einstein Weyl structures were discussed in section $8$ of \cite{Fox-2dahs}. Independently A.~D. Popov proposed such equations in \cite{Popov-vortextypeequations} and they have been called the \textit{Popov equations} by N. Manton in \cite{Manton-popovequations}; see section \ref{vortexsection} for related remarks. They make sense for sections of line bundles other than the tangent bundle; the corresponding real equations are like \eqref{vlike}, though with a completely symmetric tensor in place of $Y$ and an appropriate generalization of the Killing condition. While in both the Einstein Weyl and the vortex-like case these equations are most interesting when posed for sections of line bundles other than the tangent bundle, the particular case \eqref{vlike} considered here is interesting because it can be solved explicitly in quite elementary terms, and the solutions come in one parameter families solving the Ricci flow. A special case of this last statement, explained in section \ref{solsection} is that any steady gradient Ricci soliton determines a solution to the $\varep = 1$ case of \eqref{vlike}. Lemma \ref{ricsollemma}, shows conversely that solutions of \eqref{vlike} for which there vanishes the invariant defined in \eqref{spdefined} are steady gradient Ricci solitons. This suggests a relation between the real vortex equations and the Ricci flow. In particular, it suggests the real vortex equations can be viewed as the fixed points of some natural flow on the moduli space of pairs $(h, Y)$. It would be interesting to give substance to such a speculation.

When the Ricci flow $\hh$ beginning at the metric $h$ of a solution $(h, Y)$ to the real vortex equations is unique, then the Killing property of $Y$ is preserved along the flow. While it is not obvious that the condition \eqref{realvortex} is also preserved by the Ricci flow, Theorem \ref{flowtheorem} shows that locally there is a unique Ricci flow such that this is true for some $\tau$ depending on the flow parameter $t$, and Corollary \ref{flowcorollary} shows that it is true globally on a compact surface. 
Thus the Ricci flow can somehow be regarded as a flow obtained by varying the vortex parameter, although a precise formulation of this statement and a conceptual explanation for it are not given here. The results lend credence to the idea that the Ricci flow is related to some natural flow obtained from the moduli space of solutions to some vortex-like equations by varying the vortex parameter. In \cite{Manton-vortexricci}, N. Manton has given indications of a possible relation obtained by considering the metric on the moduli space of one vortices on a compact Riemann surface as a function of the vortex parameter and the Ricci flow on that surface. Although no direct connection between Manton's ideas and those explained here is yet apparent, they have a similar spirit.  

That a pair $(\hh, Y)$ comprising a Ricci flow $\hh$ and an $\hh$-Killing field $Y$ solves \eqref{vlike} can be reduced to the pair of equations \eqref{wss} and \eqref{lde2}, and these can be integrated explicitly. Section \ref{metricssection} is devoted to describing in detail the Ricci flows that result. They include a number of well known examples. In particular they include steady gradient Ricci solitons, such as the cigar soliton; the Fateev-Onofri-Zamolodchikov-King-Rosenau sausage metrics on the sphere; and a pair of solutions to the Ricci flow, one on the sphere and one on the torus, that were found in \cite{Fox-2dahs} in connection with Einstein-Weyl structures. Several of the metrics in section \ref{metricssection} have been considered previously by I. Bakas in \cite{Bakas} and can be found by using the ansatz used in \cite{Fateev-Onofri-Zamolodchikov} to find the sausage metric. Among the metrics constructed there are immortal, ancient, and eternal Ricci flows, and there are also constructed several examples having conical singularities at the zeros of $Y$. While the examples that were not already well known mostly have some undesirable properties, e.g. curvature blowup, it is interesting that among the fairly limited class of metrics solving the real vortex equations appear many of the most interesting Ricci flows on surfaces. 

\section{Real vortex equations}
\subsection{}\label{weylsection}
A pair $(\nabla, [h])$ comprising a torsion-free affine connection $\nabla$ and a conformal structure $[h]$ is a \textit{Weyl structure} if for each $h \in [h]$ there is a one-form $\ga_{i}$ such that $\nabla_{i} h_{jk} = 2\ga_{i}h_{jk}$.  Here, and where convenient in all that follows, the abstract index conventions are used; in particular grouping of indices between parentheses (resp. square brackets) indicates complete symmetrization (resp. anti-symmetrization) over the enclosed indices. The one-form $\ga$ is the \textit{Faraday primitive} associated to $h \in [h]$, and $F = -d\ga$ is the \textit{Faraday curvature}. Since when $h$ is rescaled conformally $\ga$ changes by addition of an exact one-form, $F$ depends only on the pair $(\nabla, [h])$ and not on the choice of $h \in [h]$. Note that $\ga$ depends only on the homothety class of $h$ and not on $h$ itself. An $h \in [h]$ for which the associated Faraday primitive $\ga$ is coclosed is called a \textit{distinguished} representative of $[h]$. From the Hodge decomposition it follows that if the underlying manifold is compact there is a distinguished representative determined uniquely up to homothety.

In dimensions greater than two a Weyl structure is said to be \textit{Einstein} if the trace-free symmetric part of the Ricci tensor of $\nabla$ vanishes. However, since on a surface such a condition is automatic, the situation for surfaces is similar to that for ordinary metrics on surfaces, where the correct analogue of the Einstein condition is constant scalar curvature. In \cite{Calderbank-mobius} and \cite{Calderbank-twod}, D. Calderbank defined a Weyl structure $(\nabla, [h])$ on a surface to be \textit{Einstein} if it satisfies the equation
\begin{align}\label{cew}
\begin{split}
0 &= \nabla_{i}(|\det h|^{1/2}(\sR_{h} - 2\dad_{h}\ga)) + 2|\det h|^{1/2}h^{pq}\nabla_{p}F_{iq} \\
& = |\det h|^{1/2}\left(d_{i}(\sR_{h} - 2\dad_{h}\ga) + 2\ga_{i}(\sR_{h} - 2\dad_{h}\ga) + 2h^{pq}\nabla_{p}F_{iq} \ \right)
\end{split}
\end{align}
where $h$ is any representative of $h$, $h^{ij}$ is the symmetric bivector inverse to $h_{ij}$, and $\dad_{h}$ is the adjoint of the exterior differential corresponding to $h$. The equation \eqref{cew} is conformally invariant, for if $\tilde{h} = fh$ then $f(\sR_{h} - 2\dad_{\tilde{h}}\tilde{\ga}) = \sR_{h} - 2\dad_{h}\ga$, where $\tilde{\ga}$ is the Faraday primitive associated to $\tilde{h}$. The quantity $\sR_{h} - 2\dad_{\tilde{h}}\tilde{\ga}$ arises as the $h$-trace of the Ricci curvature of $\nabla$. By the following theorem, on a compact surface, Calderbank's Einstein Weyl condition is equivalent to the $\varep = -1$ case of the real vortex equations \eqref{vlike}. Calderbank's original definition is Definition $3.2$ of \cite{Calderbank-mobius}; see also Corollary $3.4$ of that same paper, the conclusion of which was taken as the definition in Definition $6.2$ of \cite{Fox-2dahs}. For the proof of Theorem \ref{ewtheorem} see Theorem $3.7$ of \cite{Calderbank-mobius} or Theorem $7.1$ of \cite{Fox-2dahs}.

\begin{theorem}[\cite{Calderbank-mobius}]\label{ewtheorem}
A Weyl structure $(\nabla, [h])$ on a compact surface is Einstein if and only if for any distinguished metric $h \in [h]$ with associated Faraday primitive $\ga$, the vector field $Y^{i} = h^{ip}\ga_{p}$ metrically dual to $\ga$ is $h$-Killing and with $h$ constitutes a solution to the real vortex equations \eqref{vlike} in the case $\varep = -1$, that is $d(\sR_{h} - 4|Y|_{h}^{2}) = 0$.
\end{theorem}

\noindent
Here it is convenient to take the $\varep = -1$ case of \eqref{vlike} as the definition of an Einstein Weyl structure. Given a solution $(h, Y)$ of \eqref{vlike} with $\varep = -1$, the Weyl connection of the associated Einstein Weyl structure is $\nabla = D - 2\ga_{(i}\delta_{j)}\,^{k} + h_{ij}Y^{k}$, where $\ga_{i} = Y^{p}h_{ip}$ and $D$ is the Levi-Civita connection of $h$.

\subsection{}
\label{vortexsection}
A \textit{Riemann surface} means a one-dimensional complex manifold. It is equivalent to specify a conformal structure $[h]$ and an orientation, in which case the complex structure $J$ is the unique one compatible with $[h]$ and the given orientation. On a Riemann surface a real vector field $Y$ is conformal Killing if and only if its $(1, 0)$ part $Y^{(1,0)}$ is a holomorphic vector field. In particular, the only orientable compact surfaces possibly supporting nontrivial solutions to \eqref{vlike} are the sphere and torus. In the Einstein Weyl case, all such structures have been described in various forms in \cite{Calderbank-mobius}, \cite{Calderbank-twod}, and section $10$ of \cite{Fox-2dahs}. 

On a Riemann surface with complex structure $J$ let $h$ be a Riemannian metric representing the conformal structure and having K\"ahler form $\om$, and let $\lne$ be a smooth complex line bundle over $M$ with a fixed Hermitian metric $m$. The \textit{Abelian vortex} or \textit{Abelian Higgs} equations on a compact Riemann surface are a modification of the Ginzburg-Landau model for superconductors first studied by M. Noguchi, \cite{Noguchi}, and S. Bradlow, \cite{Bradlow}, (see \cite{Jaffe-Taubes} for background and context).
The Abelian vortex equations with parameter $\tau$ are the following equations for a pair $(\nabla, s)$ comprising a Hermitian connection $\nabla$ on $\lne$ and a smooth section $s$ of $\lne$.
\begin{align}\label{vortexeqns}
&\delbar_{\nabla}s = 0,& & 2 \j \dlef(\Om) + |s|_{m}^{2} = \tau.
\end{align}
Here $\Om$ is the curvature of $\nabla$, viewed as a real-valued two-form on $M$; $\delbar_{\nabla}$ is the $(0,1)$ part of $\nabla$; and $\dlef$ is the dual Lefschetz operator on $(1,1)$ forms, normalized so that $\dlef(\om) = 1$. The first equation says that $\delbar_{\nabla}$ is a holomorphic structure on $\lne$ with respect to which $s$ is a holomorphic section, while the second equation is something like an Einstein equation. (Note that there is no need to include the condition $\Om^{(0,2)} = 0$, as it is automatic on a Riemann surface). A solution of \eqref{vortexeqns} is \textit{nontrivial} if $s$ is not identically zero. The trivial solutions correspond to holomorphic structures on $\lne$; a precise statement is Theorem $4.7$ of \cite{Bradlow}.

The modification of \eqref{vortexeqns} to be considered here consists in the equations
\begin{align}\label{signedvortexeqns}
&\delbar_{\nabla}s = 0,& & 2 \j \dlef(\Om) + \varep |s|_{m}^{2} = \tau,
\end{align}
in which all the data is as in \eqref{vortexeqns}, and $\varep$ is one of $\pm 1$. The $\varep = +1$ case simply yields \eqref{vortexeqns}. The $\varep = -1$ case will be called here the \textit{signed Abelian vortex equations}. 

In \cite{Garcia-Prada-invariantconnections}, O. Garcia-Prada generalized arguments of C. Taubes, \cite{Taubes-equivalence}, and E. Witten, \cite{Witten-multipseudoparticle}, to show that the Abelian vortex equations for a line bundle on a surface $M$ are obtained from the Hermitian Yang-Mills equations on a rank two holomorphic vector bundle over the product of $M$ with $\proj^{1}(\com)$ via dimensional reduction utilizing the $SU(2)$ invariance of the K\"ahler metric on $M\times \proj^{1}(\com)$. In \cite{Popov-vortextypeequations}, A.~D. Popov obtains the signed vortex equations on $S^{2}$ by an analogous dimensional reduction of the Yang-Mills equations on the product of $S^{2}$ with hyperbolic space exploiting the $SU(1, 1)$ invariance of the metric on the product. These results suggest that the equations \eqref{signedvortexeqns} are equally natural with either sign.

Solutions to the Abelian vortex equations are considered equivalent if they are related by the action of the unitary gauge group ($S^{1}$-valued smooth functions) on the space of pairs $(\nabla, s)$ comprising a Hermitian connection and a smooth section of $\lne$, and the moduli space of solutions means the quotient of the space of pairs solving \eqref{vortexeqns} by the action of the unitary gauge group. 

The basic theorem about the Abelian vortex equations on a surface is the following.
\begin{theorem}[\cite{Noguchi}, \cite{Bradlow}, \cite{Garcia-Prada}]\label{vortextheorem}
Let $M$ be a compact surface equipped with a K\"ahler metric $(h, J)$. Let $\lne$ be a smooth complex line bundle with a fixed Hermitian metric $m$. Let $D$ be an effective divisor of degree equal to $\deg(\lne)$. There exists a nontrivial solution $(s, \nabla)$ of the vortex equations \eqref{vortexeqns}, unique up to unitary gauge equivalence, if and only if $4\pi\deg(\lne) < \tau \vol_{h}(M)$. Moreover the holomorphic line bundle and section canonically associated to $D$ are $(\lne, \delbar_{\nabla})$ and $s$.
\end{theorem}

The number $\deg \lne$ is called the \textit{vortex number} because the section $s$ has $\deg \lne$ zeros (counted with multiplicity), which are regarded as \textit{vortices}. For the same reason, if $\deg \lne = N$, a solution $(h, s)$ is referred to as an $N$-vortex solution, with or without multiplicity as the zeros of $s$ are not or are distinct. The space of effective divisors on $M$ of a given degree $r$ is the symmetric product $S^{r}(M)$ of $M$ and Theorem \ref{vortextheorem} shows that $S^{\deg(\lne)}(M)$ is in bijection with the moduli space of unitary gauge equivalence classes of vortex solutions on $\lne$. It is shown in \cite{Garcia-Prada} by symplectic reduction, that this moduli space carries a K\"ahler structure. 

As is explained in \cite{Bradlow}, a problem equivalent to solving \eqref{vortexeqns} consists in finding the Hermitian metric $m$ for which \eqref{vortexeqns} hold, given a holomorphic line bundle $\lne$ on a K\"ahler surface and a prescribed section $s$ of $\lne$. From this point of view the complex gauge group of $\lne$ (comprising smooth functions $g: M \to \com^{\ast}$) acts by pushforward on holomorphic structures and holomorphic sections, and on Hermitian metrics by multiplication by a factor $|g|^{2}$. Proposition $3.7$ of \cite{Bradlow} shows that given a complex line bundle $\lne \to M$ over a K\"ahler manifold $M$, the moduli space of unitary gauge equivalence classes of pairs $(\nabla, s)$ solving \eqref{vortexeqns} with respect to a fixed Hermitian metric $m$ on $\lne$ is in bijection with the quotient modulo the action of the group of complex gauge transformations of the space of triples $(\delbar_{\lne}, s, k)$ comprising a holomorphic structure $\delbar_{\lne}$ on $\lne$, a holomorphic section $s$ of $\lne$, and a Hermitian metric $k$ on $\lne$ solving \eqref{vortexeqns}. The image of the action of $g:M \to \com^{\ast}$ on $(\delbar_{\lne}, s, k)$ is $(g^{-1}\circ \delbar_{\lne}\circ g, g^{-1}s, |g|^{2}k)$. The bijection assigns to the equivalence class $[\nabla, s]$ the equivalence class $[\nabla, s, m]$. As will be explained in section \ref{equivalencesection}, this way of viewing \eqref{vortexeqns} is the most relevant for the relation with the real equations \eqref{vlike}.

On a Riemann surface, a $q$-differential is a smooth section of the $q$th power $\cano^{q}$ of the complex cotangent bundle. The real part of a $q$-differential $s$ is a trace-free symmetric $q$-tensor $X$, covariant or contravariant according to whether $q$ is positive or negative, and so $s$ can be written as the $(|q|, 0)$ part of the real tensor $X$, $s = X^{(|q|, 0)}$. That $s$ be holomorphic is equivalent to $X$ being a Codazzi tensor, for $q > 1$; to $X$ being harmonic, for $q = 1$; and to $X$ being a conformal Killing tensor, for $q$ negative (see Lemma $3.5$ of \cite{Fox-2dahs}).

Let $\lne = \cano^{q}$ for some $q \in \integer$. In the equations \eqref{signedvortexeqns} there is no \textit{a priori} relation between the K\"ahler structure $(h, J)$ on $M$ and the Hermitian metric $m$ and holomorphic structure on $\lne$. However, since $\lne$ is a power of the complex tangent bundle, it makes sense to speak of the Hermitian metric induced on $\lne$ by $h$, and so it makes sense to consider solutions in which $m$ is this induced Hermitian metric and the holomorphic structure on $\lne = \cano^{q}$ is the standard one. In what follows this will be case of primary interest. In this case a solution $s$ of \eqref{signedvortexeqns} is a holomorphic $q$-differential and $\nabla$ is the connection induced on $\lne$ by the Levi-Civita connection $D$ of $h$; the corresponding divisor is canonical, so the solution of \eqref{signedvortexeqns} will be said to be \textit{canonical} as well. Precisely, a \textit{canonical solution} of the (signed) vortex equations comprises a Riemann surface $(M, J)$, a holomorphic $q$-differential $s$, and a metric $h$ representing the given conformal structure such that 
\begin{align}\label{sve}
2 \j \dlef(\Om) + \varep |s|_{h}^{2} = \tau,
\end{align}
for some constant $\tau$. Here $\Om  = (\j q \sR_{h}/2)\om_{h}$ is the curvature of the Hermitian connection induced on $\cano^{q}$ by the Levi-Civita connection $D$ of $h$ and $\dlef$ is defined in terms of the K\"ahler structure $(h, J)$ so that $\dlef(\Omega) = (\j q/2)\sR_{h}$. Write $s = 2^{1 - q/2}|q|^{1/2}X^{(|q|, 0)}$. Then, since $2|X^{(|q|, 0)}|_{h}^{2} = |X|_{h}^{2}$,
\begin{align}\label{rsve}
2\j \dlef(\Omega) + \varep |s|_{h}^{2} = -q \left(\sR_{h} - \varep \sign(q) 2^{1-q} |X|^{2}_{h}\right),
\end{align}
so that $s$ solves \eqref{sve} if and only if $\sR_{h} - \varep \sign(q)2^{1-q}|X|_{h}^{2}$ is constant. Thus a canonical solution of the (signed) vortex equations is equivalent to a pair comprising a Riemannian metric $h$ and a trace-free Codazzi or conformal Killing tensor $X$ satisfying
\begin{align}\label{signedew}
d(\sR_{h} - \varep \sign(q)2^{1-q}|X|_{h}^{2}) = 0.
\end{align}
If the Gauss-Bonnet theorem is valid, e.g. if $M$ has finite topological type, finite volume, and integrable curvature, and there is a solution to \eqref{signedew}, there must hold
\begin{align}
4\pi \chi(M) = \varep \sign(q)2^{1-q}||X||_{h}^{2} + \tau \vol_{h}(M),
\end{align}
where $\tau$ is the constant value of $\sR_{h} - \varep \sign(q)2^{1-q}|X|_{h}^{2}$. This shows that the condition 
\begin{align}
&\varep\sign(q)\left(\tau - \tfrac{4\pi \chi(M)}{\vol_{h}(M)}\right) \leq 0,
\end{align}
on $\tau$ is necessary for the existence of solutions.


Note that by themselves the equations \eqref{signedew} for $(q, \ep)$ and $(-q, -\ep)$ are the same up to a power of $2$ that can be absorbed into the section $X$; what changes with the change in parameters is the condition (Codazzi or conformal Killing) imposed on the section $X$. For example, in the $q = \pm 1$ cases, it is different to demand that a vector field be the real part of a holomorphic vector field and that its dual one-form be the real part of a holomorphic differential; the former imposes that the vector field be conformal Killing while the latter imposes that the dual one-form be harmonic.

The case most important here is $q = -1$, in which case $s$ is a holomorphic vector field and so $X$ is a conformal Killing field. In this case solutions of the real vortex equations \eqref{realvortex} give rise to solutions of the signed Abelian vortex equations. However, not even all canonical solutions of the signed Abelian vortex equations arise in this way because it is not the case that every holomorphic vector field is the $(1,0)$ part of a real Killing field. The real version \eqref{signedew} of \eqref{signedvortexeqns} is then \eqref{realvortex}, but with $Y$ only required to be conformal Killing. 
 
In general, if a symmetric trace-free $|q|$-tensor $X$ is given such that $X^{(|q|, 0)}$ is holomorphic, then the equation \eqref{signedew} can be solved as follows. Let $\tilde{h}$ be the unique metric conformal to $h$ with scalar curvature contained in $\{0, 2, -2\}$ and write $h = e^{u}\tilde{h}$. The equation \eqref{signedew} becomes
\begin{align}\label{signedew2}
\lap_{\tilde{h}}u - \sR_{\tilde{h}} + \tau e^{u} + \varep\sign(q)2^{1-q}e^{(1-q)u}|X|_{\tilde{h}}^{2} = 0.
\end{align}
The solvability or no of \eqref{signedew2} on a compact surface of genus at least two is usually ascertainable on general grounds, while on spheres, tori, and some noncompact surfaces it is less straightforward. Particularly for the sphere and torus it can be more convenient to use as the background metric the singular flat metric $\sth = |X|^{2/q}h$ instead of $\tilde{h}$. For example, on a surface of genus at least two, the case $\varep = -1$, $q \geq 1$ of \eqref{signedew2} can always be solved (see Corollary $9.1$ of \cite{Fox-2dahs}) provided $\tau$ is negative. For another example, in the $\varep = -1$ and $q = -1$ case $X$ is a Killing field, and using this fact the analogue of \eqref{signedew2} with $\sth$ in place of $\tilde{h}$ reduces to an ordinary differential equation which can be solved straightforwardly. Although cast in different terms, this is essentially the approach taken in section \ref{metricssection}.

The equation \eqref{signedew2} is slightly more complicated than the similar equation arising for the vortex equations, e.g. equation $4.1$ of \cite{Bradlow}. The difference is the requirement that the metric on $\lne$ be that induced by the metric on the underlying surface. This condition introduces an extra term in \eqref{signedew2}, with the consequence that solvability of \eqref{signedew2} does not reduce directly to the well known results of Kazdan-Warner, \cite{Kazdan-Warner}. Other cases of \eqref{signedew2} have been studied previously. For example, the case with $\varep = 1$, $q = 2$, and $\tau < 0$ arises as the Gauss equation for a minimal surface in a hyperbolic three manifold; see Theorem $4.2$ of \cite{Uhlenbeck-minimal}, in which this equation plays an important role.

\subsection{}\label{equivalencesection}
The standard notions of isomorphism of Einstein Weyl structures and gauge equivalence of solutions of the abelian vortex equations do not lead to the same notions of equivalence of solutions to the real vortex equations \eqref{vlike}, so some discussion of such notions is necessary. Clearly two solutions $(h, Y)$ and $(\bar{h}, \bar{Y})$ of \eqref{vlike} related by a diffeomorphim should be considered isomorphic. In general the canonical triples $(\delbar_{\nabla}, s, h)$ representing the canonical solutions of the signed Abelian vortex equations corresponding to a pair $(h, Y)$ and its image under a biholomorphism of the complex structure determined by $h$ and the orientation of the underlying surface need not be equivalent modulo a complex gauge transformation; that is, they might determine different points in the vortex moduli space. Usually solutions of the Abelian vortex equations related by a biholomorphisms are \textit{not} identified, because the zeros of $s$ are regarded as vortices, and their absolute positions are considered meaningful. On the other hand, the notion of equivalence relevant for the Abelian vortex equations is complex gauge equivalence. How this translates for pairs $(h, Y)$ solving \eqref{vlike} is described now. The solution $(\delbar, s, h)$ of \eqref{signedvortexeqns} corresponding to $(h, Y)$ is canonical. That a solution of \eqref{signedvortexeqns} be canonical is not preserved by the action of the complex gauge group. The statement that the image $g\cdot(\delbar, s, h) = (g^{-1}\circ \delbar \circ g , g^{-1}s, |g|^{2}h)$ of $(\delbar, s, h)$ under the action of a complex gauge transformation $g:M \to \com^{\ast}$ be again canonical admits two possible interpretations. The more restricted interpretation stems from considering that a gauge transformation disassociates the base and fiber metrics, acting only on the latter, which entails considering that  $g\cdot(\delbar, s, h)$ is canonical with respect to the fixed K\"ahler structure $(h, J)$ on $M$. With this interpretation, $\delbar$ is fixed, meaning $g^{-1}\circ \delbar \circ g = \delbar$, which is equivalent to $g$ being holomorphic, and $|g|^{2}h = h$, so that $|g|^{2} = 1$. A holomorphic function of constant norm is constant, so $g$ is constant, taking values in $\mathbb{S}^{1}$. The more liberal interpretation stems from regarding the base metric as induced by the fiber metric. This entails that $g\cdot (\delbar, s, h)$ be canonical with respect to the K\"ahler structure $(|g|^{2}h, J)$ (which determines the same Riemann surface structure). In this case it still must be that $\delbar$ is fixed, so that $g$ is holomorphic. If $M$ is compact, this already forces $g$ to be constant. However, if $M$ is noncompact, then it admits nonconstant holomorphic functions. Nonetheless, if $(\delbar, s, h)$ and $g\cdot (\delbar, s, h)$ both solve \eqref{sve}, the corresponding real pairs $(h, Y)$ and $(\tilde{h}, \tilde{Y})$ are related by $\tilde{h} = e^{2c}h$ and $\tilde{Y} = e^{-2c}(aY + bJY)$ where $g = e^{c}(a + \j b)$ and $a^{2} + b^{2} = 1$. Since $g$ is holomorphic $\lap_{h}c = 0$, so, since $\sR_{\tilde{h}} + 4\varep|\tilde{Y}|_{\tilde{h}}^{2} = e^{-2c}(\sR_{h} - 2\lap_{h}c + 4\varep|aY + bJY|^{2}_{h}) = e^{-2c}(\sR_{h} + 4\varep|Y|^{2}_{h}) = e^{-2c}\tau$ must be constant, $c$ must be constant as well. Again, that $|g|^{2}$ be constant and $g$ be holomorphic means $g$ is constant. With either interpretation the only elements of the complex gauge group acting on canonical solutions of the signed Abelian vortex equations are constants $z  = e^{c}e^{\j\theta} \in \com^{\ast}$. The corresponding action of $\com^{\ast}$ on pairs $(h, Y)$ is $z\cdot (h, Y) = (|z|^{2}h, |z|^{-2} \re (e^{-\j\theta}Y^{(1,0)})) = (e^{2c}h, e^{-2c}Y^{\theta})$, where $Y^{\theta}  = \cos \theta Y + \sin \theta JY$ is the real part of $e^{-\j\theta}Y^{(1,0)}$. In general this action does not preserve the property that $Y$ be Killing. A precise statement is the following.
\begin{lemma}\label{killinggaugelemma}
Let $M$ be a surface with a K\"ahler structure $(h, J)$ and let $Y \in \Ga(TM)$ be a Killing field for $h$. If for some $\theta \in (0, 2\pi)$ the vector field $Y^{\theta} = \cos \theta Y + \sin \theta JY$ is Killing for $h$ then either $Y$ is parallel or $\theta = \pi$ and $Y^{\theta} = -Y$.
\end{lemma}
\begin{proof}
Suppose $Y$ and $Y^{\theta}$ are Killing for some $\theta \in (0, 2\pi)$. Let $\ga$ be the one-form dual to $Y$; then $\star \ga$ is the one-form dual to $JY$. Since $Y$ is Killing, $\ga$ is coclosed. Since $Y$ and $Y^{\theta}$ are Killing, so is $\sin \theta JY = Y^{\theta} - \cos \theta Y$. If $\sin \theta \neq 0$ this means $JY$ is Killing, and so $2D\star \ga = d\star \ga = 0$, the last equality becaue $\ga$ is coclosed. In this case, $JY$ is parallel, and so $Y$ is parallel. Hence if both $Y$ and $Y^{\theta}$ are Killing for some $\theta \in (0, 2\pi)$ then either $Y$ is parallel, or $\theta = \pi$ and $Y^{\theta} = -Y$. 
\end{proof}

It follows that if $(h, Y)$ solves \eqref{vlike} then so does $\pm e^{c} \cdot (h, Y) = (e^{2c}h, \pm e^{-2c}Y)$ for all $c \in \rea$. Since in general homothetic metrics need not be diffeomorphic, it is at first not clear in what sense the solutions $e^{c}\cdot (h, Y)$ and $(h, Y)$ are equivalent. However, since neither the Levi-Civita connection of $e^{2c}h$ nor the one-form $\ga$ dual to $e^{-2c}Y$ via $e^{2c}h$ depends on $c$, the resulting Weyl connection is independent of $c$. Hence, in the case $\varep = -1$ corresponding to the Einstein Weyl equations, the solutions $e^{c}\cdot (h, Y)$ and $(h, Y)$ are equivalent in the sense that they determine the same Einstein Weyl structure. This justifies regarding $(h, Y)$ and $e^{c}\cdot (h, Y)$ as equivalent even though they are not related by a diffeomorphism, and it is natural to extend this notion of equivalence to the $\varep = 1$ case as well. Solutions of \eqref{vlike} equivalent in this sense will be said to be \textit{scaling equivalent}. As was explained above, scaling equivalence is a vestigial manifestation of the gauge equivalence (in the broader sense) of the associated canonical solutions of the signed Abelian vortex equations. It is also the case that if $(h, Y)$ solves \eqref{vlike} then so too does $(h, -Y)$, and by the preceeding these solutions could also be considered equivalent. On the other hand, while it can happen that $(h, Y)$ and $(h, -Y)$ be diffeomorphism equivalent if there is an isometry of $h$ sending $Y$ to $-Y$, in the $\varep = -1$ case, $(h, Y)$ and $(h, -Y)$ generally induce nonisomorphic Weyl structures. 

\subsection{}\label{solsection}
A \textit{Ricci soliton} on a surface is a Riemannian metric $h$ for which there are a vector field $X$ and a constant $c \in \rea$ such that $\tfrac{1}{2}\sR_{h}h_{ij} + \tfrac{1}{2}(\lie_{X}h)_{ij} = c h_{ij}$. It is a \textit{gradient Ricci soliton} if $X$ is the $h$-gradient of a smooth function $f$. In this case $\sR_{h} + \lap_{h}f = 2c$. Differentiating this and using the Ricci identity shows that $d\sR_{h} = \sR_{h}df$, from which it follows that $\sR_{h} + |df|^{2}_{h} - 2c f$ is constant. The gradient $X^{i} = h^{ip}df_{p}$ of the potential of a gradient Ricci soliton is conformal Killing, for by definition $\lie_{X}h = (2c - \sR_{h})h$. In particular, in the case the gradient Ricci soliton $h$ is \textit{steady}, meaning $c =0$, the pair $(h, \tfrac{1}{2}X)$ solves the variant of \eqref{vlike} in which the vector field may be conformal Killing. However, since $Ddf$ is symmetric, $X$ itself is Killing if and only if it is parallel, in which case $h$ has constant curvature $\sR_{h} = 2c$. On the other hand, $JX$ is always Killing, where $J$ is the complex structure determined by $h$ and a given orientation on the surface, for $2J_{j}\,^{p}D_{i}df_{p} = (\sR_{h} - 2c)\om_{ij}$, where $\om_{ij} = J_{i}\,^{p}h_{pj}$ is the K\"ahler form of $(h, J)$, so that $\lie_{JX}h = 0$. By Lemma \ref{killinggaugelemma} either $X^{\theta}  = \cos \theta X + \sin \theta JX$ is parallel for all $\theta \in [0, 2\pi)$ and $h$ has constant curvature or among the conformal Killing fields $X^{\theta}$ exactly $\pm JX$ are Killing. In the latter case, $(h, \pm \tfrac{1}{2}JX)$ are solutions of the vortex-like equations \eqref{vlike}. By Theorem $10.1$ of \cite{Hamilton-riccisurfaces}, any Ricci soliton on a compact surface has constant curvature, so no interesting examples of solutions to \eqref{vlike} arise in this way on compact surfaces. On the other hand, there are nontrivial steady Ricci solitons on noncompact surfaces, and these yield solutions of \eqref{vlike}, as is detailed in section \ref{solitoncasesection}. 

\subsection{}\label{rfsection}
This section records some general facts about solutions to \eqref{vlike} that will be used in section \ref{metricssection} to find explicit solutions. In the particular case of Einstein Weyl structures, part of the discussion appears in some equivalent form in sections $5$ and $6$ of \cite{Fox-2dahs}; see in particular Lemma $6.4$ of \cite{Fox-2dahs}.

On an oriented surface $M$ consider a pair $(h, Y)$. Let $J$ be the complex structure determined by $h$ and the given orientation, let $\om_{h}$ be the K\"ahler form of $h$, and let $D$ be the Levi-Civita connection of $h$. Define a one-form $\ga$ by $\ga = \imt(Y)h$. The Hodge star on one-forms is given by $\star \al = -\al \circ J$. Write $F = -d\ga$ and define a function $\fd_{h}$ by $2F = \fd_{h}\om_{h}$ (equivalently, $2\star F = \fd_{h}$). If $Y$ is a Killing field, then $\ga$ is coclosed, so $d\star \ga = 0$. Let $\tilde{M}$ be the universal cover of $M$. The pullbacks to $\tilde{M}$ of objects defined on $M$ will be written with the same notation. The pullback of $(h, Y)$ to $\tilde{M}$ by definition comprises the pullback of $h$ and the unique vector field on $\tilde{M}$ projecting to $Y$. On $\tilde{M}$ there is a globally defined function $\mom$ such that $d\mom = -\star \ga$. By definition $|\ga|^{2}_{h}\om_{h} = \ga \wedge \star \ga = d\mom \wedge \ga$. Interior multiplying with $Y$ shows that $|\ga|^{2}_{h}(\imt(Y)\om_{h} + d\mom) = 0$. Since $Y$ is the real part of a holomorphic vector field, its zeros are isolated points, and so the preceeding identity implies $\imt(Y)\om_{h} + d\mom = 0$ on $\tilde{M}$. When $Y$ is complete this means that $\mom$ is a moment map for the action generated by $Y$ on $\tilde{M}$, and in what follows $\mom$ will be called a moment map even if $Y$ is not assumed complete.

Now suppose $(h, Y)$ solves the real vortex equations \eqref{vlike}. Since $Y$ is Killing, $4D\ga = 2d\ga = -\fd_{h}\om_{h}$. Observe that $Y^{p}\om_{ip} = \ga_{p}J_{i}\,^{p} = -(\star \ga)_{i}$. This has the consequences,
\begin{align}\label{dyh}
D_{i}|Y|^{2}_{h} = D_{i}|\ga|_{h}^{2} = 2Y^{p}D_{i}\ga_{p} = \tfrac{1}{2}\fd_{h}(\star \ga)_{i},
\end{align}
\begin{align}\label{dstarga}
&-Dd\mom = D\star \ga = \tfrac{1}{4}\fd_{h}h,
\end{align}
the first equality in \eqref{dstarga} when $\mom$ is well defined. By \eqref{dstarga} the function $\mom$ on $\tilde{M}$ is what is called a \textit{concircular scalar field} by Y. Tashiro in \cite{Tashiro}. By Theorem $1$ of \cite{Tashiro} the number of critical points of a concircular scalar field on a complete Riemannian manifold is at most two, and, applying this to $\mom$, it follows that if the metric $h$ of a solution $(h, Y)$ of the real vortex equations is complete then $Y$ has at most two zeros. Lemma \ref{twozerolemma} below obtains a stronger conclusion with a weaker hypothesis.

Using \eqref{dstarga} it is straightforward to show that $D$ is given by
\begin{align}\label{krlc}
&D_{Y}Y = -\tfrac{1}{4}\fd_{h}JY,& &D_{JY}Y = D_{Y}JY= JD_{Y}Y = \tfrac{1}{4}\fd_{h}Y,& &D_{JY}JY = JD_{JY}Y = \tfrac{1}{4}\fd_{h}JY.
\end{align}
From \eqref{krlc} it follows that $D_{JY}JY \wedge JY = 0$, so that the integral curves of $JY$ are projective geodesics, meaning their images coincide with the images of $h$-geodesics. The unit norm vector field $U = -|Y|^{-1}_{h}JY$, defined on the open dense complement $M^{\ast}$ of the zero locus of $Y$, satisfies $D_{U}U = 0$, so that its nontrivial integral curves are $h$-geodesics. 

Differentiating the first equation of \eqref{realvortex} and using \eqref{dyh} yields
\begin{align}\label{drd}
d \sR_{h} = -4\varep d|\ga|_{h}^{2}  =  - 2\varep \fd_{h} \star \ga.
\end{align}
The full curvature of $D$ is $R_{ijkl} = R_{ijk}\,^{p}h_{pl} = \sR_{h} h_{l[i}h_{j]k}$ and so $\om^{ij}R_{ijkl} = -\sR_{h} \om_{kl}$. Applying this and the Ricci identity yields
\begin{align}
\begin{split}
D_{i}\fd_{h} & = D_{i}(\om^{pq}F_{pq}) = -\om^{pq}D_{i}d\ga_{pq} = -2\om^{pq}D_{i}D_{p}\ga_{q} = -2\om^{pq}D_{p}D_{i}\ga_{q} + 2\om^{pq}R_{i[pq]}\,^{a}\ga_{q}\\
& = 2\om^{pq}D_{[p}D_{q]}\ga_{i} - \om^{pq}R_{pqi}\,^{a} = -2\om^{pq}R_{pqi}\,^{a} = 2\sR_{h} J_{i}\,^{p}\ga_{p} = -2\sR_{h}(\star \ga)_{i},
\end{split}
\end{align}
which proves
\begin{align}\label{dfd}
d\fd_{h} = -2\sR_{h} \star \ga.
\end{align}
An immediate consequence of \eqref{dfd} is that if $(h, Y)$ is a nontrivial solution of \eqref{vlike} and $Y$ is parallel, then $h$ is flat, for that $Y$ be parallel implies that $\fd_{h} = 0$, and in \eqref{dfd} this forces $\sR_{h} = 0$, since $\ga$ never vanishes.
\begin{lemma}\label{squaresconstantlemma}
Let $(h, Y)$ solve the real vortex equations \eqref{vlike} on the oriented surface $M$ with parameters $\tau = \sR_{h} + 4\varep|Y|_{h}^{2}$ and $\varep = \pm 1$. For $\ga = \imt(Y)h$ let $\mom$ be a primitive of $-\star \ga$ on the universal cover $\tilde{M}$ of $M$. Let $\la$ be $1$ or $\j$ as $\varep$ is $1$ or $-1$. Then 
\begin{enumerate}
\item \label{ub0} The quantity
\begin{align}\label{spdefined}
\sp = \sR_{h}^{2} -\varep \fd_{h}^{2}
\end{align}
is constant on $M$ and the functions $e^{\mp2\la \mom}(\sR \pm \la \fd)$ are constant on $\tilde{M}$.
\item\label{ub2} If $\varep = 1$ and one of $\sR_{h} \pm \fd_{h}$ is not identically zero, then $\mom$ is well defined as a function on $M$. Also, each of $\sR_{h} + \fd_{h}$ and $\sR_{h} - \fd_{h}$ has a definite sign if it is not identically zero.
\item\label{ub1} If $\varep = 1$, then $\sR_{h} \leq \tau$, with equality exactly where $Y$ vanishes, and $\sR_{h}^{2} \geq \sp$, with equality exactly where $\fd_{h}$ vanishes. In particular, 
\begin{enumerate}
\item if $\sp > 0$, then $\sR_{h}$ has a definite sign and either $\sqrt{\sp} \leq \sR_{h} \leq \tau$ or $\sR_{h} \leq \min\{\tau, -\sqrt{\sp}\} < 0$; 
\item if $\sp < 0$, then $\fd_{h}$ has a definite sign.
\end{enumerate}
\item\label{ub1b} If $\varep = -1$ then $-\sqrt{\sp} \leq \sR_{h} \leq \sqrt{\sp}$, with one of the inequalities an equality exactly where $\fd_{h}$ vanishes;  $-\sqrt{\sp} \leq \fd_{h} \leq \sqrt{\sp}$, with one of the inequalities an equality exactly where $\sR_{h}$ vanishes; and $\sR_{h} \geq \tau$, with equality exactly where $Y$ vanishes.
\item\label{ub4} A critical point of $\sR_{h}$ (resp. $\fd_{h}$) is either a zero of $\fd_{h}$ (resp. $\sR_{h})$ or a zero of $Y$. In the latter case it is also a critical point of $\fd_{h}$ (resp. $\sR_{h}$).
\item\label{ub5} There hold $\lap_{h}\sR_{h} + \sR_{h}^{2} + \sR_{h} = \tau + \sp$ and $\lap_{h}\fd_{h} = \tau \fd_{h} - 2\sR_{h}\fd_{h}$.
\item\label{fd1} Let $M^{\ast}$ be the complement in $M$ of the set of zeros of $Y$. If $\sp \neq 0$ then $\{p \in M^{\ast}:\fd_{h}(p) = 0\}$ is a union of closed images of $h$-geodesic integral curves of $Y$.
\end{enumerate}
\end{lemma}
\begin{proof}
Differentiating $\sp$ using \eqref{drd} and \eqref{dfd} yields
\begin{align}
\begin{split}
d(\sR^{2}_{h} -\varep \fd^{2}_{h}) &= 2\sR_{h} d\sR_{h} - 2\varep\fd_{h} d\fd_{h}  = -4\varep \sR_{h} \fd_{h} \star \ga + 4\varep \fd_{h} \sR_{h} \star \ga = 0,
\end{split}
\end{align}
so that $\sp$ is constant. By \eqref{drd} and \eqref{dfd},
\begin{align}
\begin{split}
d\left(e^{-2\la \mom}(\sR_{h} + \la \fd_{h})\right) & = e^{-2\la\mom}\left(d\sR_{h} + \la d\fd_{h}  + 2\la(\sR_{h} + \la \fd_{h})\star\ga\right)\\
& =  e^{-2\la\mom}\left(d\sR_{h} + 2\varep \fd_{h} \star \ga + \la (d\fd_{h} -2 \sR_{h}\star \ga)\right) = 0,
\end{split}
\end{align}
so $e^{-2\la \mom}(\sR_{h} + \la \fd_{h})$ is constant on $\tilde{M}$. Since $(h, -Y)$ also solves \eqref{vlike} with the parameters $\tau$ and $\varep$, and $-\mom$ is a moment map for $(h,-Y)$, $e^{2\la\mom}(\sR_{h} - \la \fd_{h})$ is also constant on $\tilde{M}$. The function $\mom$ is determined only up to addition of a constant, and it is straightforward to check that if $\sp \neq 0$ there is always a choice of $\mom$ such that $\sR_{h} + \la \fd_{h} = ce^{2\la \mom}$ and $\sR_{h} - \la \fd_{h} = \sign(\sp)ce^{-2\la\mom}$ where $c\in \rea$ satisfies $c^{2} = \sp$. When $\la = 1$ this implies $\mom$ is well defined on $M$. If both $\sR_{h}\pm \fd_{h}$ vanish, then $h$ is flat and $Y$ is parallel. Otherwise, there is a nonzero $c\in \rea$ such that $\mom$ equals one of $\pm (1/2)\log(c^{-1}(\sR_{h} \pm \fd_{h}))$, and so descends to $M$. When $\varep = 1$, the constancy of each of $e^{\mp 2\mom}(\sR_{h} \pm \fd_{h})$ means that each of $\sR_{h} \pm \fd_{h}$ has a definite sign if it is not identically zero. This shows \eqref{ub2}. 

If $\varep = 1$, then $\sR_{h} = \tau - 4|Y|^{2}_{h} \leq \tau$, with equality exactly where $Y$ vanishes, and $\sR_{h}^{2} = \sp + \fd_{h}^{2} \geq \sp$, with equality exactly where $\fd_{h}$ vanishes. The remainder of \eqref{ub1} follows straightforwardly. If $\varep = -1$ then $\sR_{h}^{2} = \sp - \fd_{h}^{2} \leq \sp$, with equality exactly where $\fd_{h}$ vanishes (and similarly with $\sR_{h}$ and $\fd_{h}$ interchanged), and $\sR_{h} = \tau + 4|Y|^{2}_{h} \geq \tau$, with equality exactly where $Y$ vanishes; this shows \eqref{ub1b}. Claim \eqref{ub4} follows from \eqref{drd} and \eqref{dfd}. Differentiating \eqref{drd} and substituting \eqref{dfd} and \eqref{dstarga} in the result yields
\begin{align}\label{ddsr}
\begin{split}
Dd\sR_{h} &= -2\varep d\fd_{h}\tensor \star \ga - 2\varep \fd_{h} D\star \ga= 4\varep \sR_{h}\star \ga \tensor \star \ga - (\varep/2)\fd_{h}^{2}h.
\end{split}
\end{align}
Tracing \eqref{ddsr} yields $\lap_{h}\sR_{h} = 4\varep\sR_{h}|Y|^{2}_{h} - \varep \fd_{h}^{2}$ which is equivalent to the first identity of \eqref{ub5}. Differentiating \eqref{dfd} and substituting \eqref{drd} and \eqref{dstarga} in the result yields
\begin{align}\label{ddfd}
\begin{split}
Dd\fd_{h} &= -2 d\sR_{h}\tensor \star \ga - 2 \sR_{h} D\star \ga= 4\varep \fd_{h}\star \ga \tensor \star \ga - (1/2)\fd_{h}\sR_{h}h.
\end{split}
\end{align}
Tracing \eqref{ddfd} yields $\lap_{h}\fd_{h} = 4\varep\fd_{h}|Y|^{2}_{h} -  \fd_{h}\sR_{h}$ which is equivalent to the second identity of \eqref{ub5}.

If $\sp \neq 0$, then $\fd_{h}(p) = 0$ implies $\sR_{h}(p) \neq 0$ and by \eqref{dfd} this implies $d\fd_{h}(p)\neq 0$ if $p \in M^{\ast}$. Thus $\{p \in M^{\ast}: \fd_{h}(p)\} = 0$ is a union of smooth closed one-dimensional submanifolds of $M^{\ast}$. Suppose $\fd_{h}(p) =0$ for some $p \in M^{\ast}$. Let $\nu:I \to \rea$ be a maximal integral curve of $Y$ such that $\nu(0) = p$. Then $\tfrac{d}{dt}\fd_{h}(\nu(t)) = d\fd_{h}(Y_{\nu(t)}) = 0$. Since $\fd_{h}(\nu(0)) = \fd_{h}(p) = 0$ this means $\fd_{h}(\nu(t)) = 0$ for all $t \in I$. By \eqref{krlc}, $4D_{\dot{\nu}}\dot{\nu} = - \fd_{h}(\nu(t))Y_{\nu(t)} = 0$, so $\nu(t)$ is an $h$-geodesic. This shows \eqref{fd1}. 
\end{proof}
Examples to be given later show that $\sR_{h}$ can be unbounded from below.


Rewriting \eqref{dstarga} yields
\begin{align}\label{ddmom}
\pm 2Dd\mom  + \tfrac{1}{2}\sR_{h}h = \tfrac{1}{2}(\sR_{h} \pm \fd_{h})h,
\end{align}
so that $h$ is a gradient Ricci soliton with potential $\pm 2\mom$ if $\sR_{h} \pm \fd_{h}$ is constant. Lemma \ref{ricsollemma} shows that \eqref{ddmom} means that if $\sp = 0$ then $h$ is a steady gradient Ricci soliton with potential one of $\pm 2 \mom$. 

\begin{lemma}\label{ricsollemma}
For a solution $(h, Y)$  of the real vortex equations \eqref{vlike}, the following are equivalent:
\begin{enumerate}
\item\label{sp1} $\sp = 0$. 
\item\label{sp2} One of $\sR_{h} \pm \fd_{h}$ is constant.
\item\label{sp3} One of $\sR_{h} \pm \fd_{h}$ vanishes identically.
\end{enumerate}
If there hold \eqref{sp1}-\eqref{sp3} then:
\begin{enumerate}
\setcounter{enumi}{3}
\item\label{rs1} Either $\varep = -1$, $h$ is flat, and $Y$ is parallel or $\varep = 1$ and $h$ is a steady gradient Ricci soliton with potential $\pm 2 \mom$, as $\sR_{h} \pm \fd_{h}$ is constant.
\item \label{fd2} If $Y$ is not parallel, then $\fd_{h}$ and $\sR_{h}$ do not vanish on $M$.
\item \label{fd3} If $Y$ has a zero, then there holds one of the following: $h$ is flat and $Y$ is identically zero; $\max_{M}\sR_{h} = \tau < 0$; or $0<\sR_{h} \leq \tau = \max_{M}\sR_{h}$.
\end{enumerate}
\end{lemma}
\begin{proof}
Obviously \eqref{sp3} implies \eqref{sp2}. Suppose there holds \eqref{sp2}. Then, by \eqref{drd} and \eqref{dfd},
\begin{align}\label{spint}
0 = d(\sR_{h} \pm \fd_{h}) = \mp 2(\sR_{h} \pm \ep \fd_{h})\star \ga.
\end{align}
Pairing \eqref{spint} with $\star \ga$ shows that $(\sR_{h} \pm \ep \fd_{h})|\ga|_{h}^{2} = 0$, and since $\ga$ has isolated zeros this forces $\sR_{h} \pm \ep \fd_{h} = 0$. If $\ep = 1$ this gives $\sR_{h}  = \mp \fd_{h}$, and so $\sp = \sR_{h}^{2} - \fd_{h}^{2} = 0$. If $\ep = -1$ it gives $\sR_{h} = \pm \fd_{h}$, which means that $2\sR_{h} = \sR_{h} \pm \fd_{h}$ is a constant. As $(h, Y)$ solves \eqref{realvortex}, this means that $|\ga|_{h}^{2}$ is constant. If $\ga$ has a zero, then it must be identically $0$, in which case $\fd_{h} = 0$, and so also $\sR_{h} = 0$ and $\sp = 0$. If $\ga$ has no zero, since $Y$ is Killing, there holds $\sR_{h} = - \lap_{h}\log |Y|^{2}_{h}$ and this vanishes since $|Y|^{2}_{h}$ is constant, so $h$ is flat and $\pm \fd_{h} = \sR_{h} = 0$, and hence also $\sp = 0$. This shows that \eqref{sp2} implies \eqref{sp1}. Suppose $\sp = 0$. If $\varep = -1$ then, by definition, $0 = \sp = \sR_{h}^{2} + \fd_{h}^{2}$, so $\sR_{h} = 0 = \fd_{h}$. Hence $h$ is flat and $Y$ is parallel. If $\varep = 1$, then $\sR_{h}^{2} = \fd_{h}^{2}$. In particular one of $\sR_{h} \pm \fd_{h}$ vanishes. This shows that \eqref{sp1} implies \eqref{sp3}. 

Now suppose there hold \eqref{sp1}-\eqref{sp3}. As $\sR_{h} = \mp \fd_{h}$, by \eqref{ddmom}, $h$ is a steady gradient Ricci soliton with potential one of $\pm 2 \mom$. If $\varep = -1$ then $0 = \sp = \sR^{2}_{h} + \fd_{h}^{2}$ implies $h$ is flat and $Y$ is parallel, and the conclusion of \eqref{fd3} follows. This shows \eqref{rs1}. If $\sp = 0$ and $Y$ is not parallel, then, by \eqref{rs1} it can be supposed $\varep = 1$, and so, by \eqref{sp1}-\eqref{sp3}, one of $\sR_{h} \pm \fd_{h}$ is identically zero and the other equals $2\sR_{h}$. By \eqref{ub2} of Lemma \ref{squaresconstantlemma}, $2\sR_{h}$ is either identically zero or never zero. Since $Y$ is not parallel, $\fd_{h}$ is not identically zero. Since $\sR_{h}^{2} = \fd_{h}^{2}$, $\sR_{h}$ is not identically zero, and so, by the preceeding, neither $\fd_{h}$ nor $\sR_{h}$ vanishes on $M$. This shows \eqref{fd2}. If $\sp = 0$, $\varep = 1$, and $Y$ vanishes at $p \in M$, then, since $\tau = \sR_{h} + 4|Y|^{2}_{h}$ is constant, $\tau = \sR_{h}(p) = \max_{M}\sR_{h}$. As before, at least one of $\sR_{h} \pm \fd_{h}$ is identically zero, and the other equals $2\sR_{h}$, which must be either nowhere zero or identically zero. In the latter case $h$ is flat and $\fd_{h} = 0$, so $Y$ is parallel, so either has no zero or is identically zero. If $h$ is not flat, then $\sR_{h}$ has a definite sign. If the sign is negative, since $\max_{M}\sR_{h} = \sR_{h}(p) = \tau$, there holds $\sR_{h} \leq \tau < 0$, while if the sign is positive there holds $0<\sR_{h} \leq \tau$. This shows \eqref{fd3}.
\end{proof}

Since, by Theorem $10.1$ of \cite{Hamilton-riccisurfaces}, a gradient Ricci soliton on a compact surface has constant curvature, on a compact surface the metric $h$ of a solution $(h, Y)$ of \eqref{vlike} with $\sp = 0$ has constant curvature. By \eqref{realvortex} this implies that $Y$ has constant norm, so either $Y$ is identically zero, or $Y$ has no zeros. In the latter case the argument in the proof of Lemma \ref{ricsollemma} shows that $h$ must be flat and $Y$ must be parallel. However, it will be seen that in the noncompact case gradient Ricci solitons give rise to nontrivial solutions of the equations \eqref{vlike}.

Let $(h, Y)$ solve \eqref{vlike} with parameters $\tau$ and $\varep$ and let $\sp$ be the constant \eqref{spdefined}. Define a constant $\rho$ by $4\varep \rho = \tau^{2} -\sp$.  Substracting the square of $\sR_{h} = \tau - 4\varep|Y|^{2}_{h}$ from $\sp = \sR^{2}_{h} - \varep \fd_{h}^{2}$ yields
\begin{align}\label{sitau}
4\rho = \fd_{h}^{2} + 8\tau |Y|^{2}_{h} - 16\varep|Y|_{h}^{4}.
\end{align}
Lemma \ref{uniquenesslemma} will show that, in a sense to be made precise, the sign of $\sp$ and the numerical value of $\rho$ are preserved by the Ricci flow. In the remainder of the present section the number of zeros of $Y$ will be related to the signs of $\sp$ and $\rho$, and there will be shown that in fact $Y$ has at most two zeros in $M$ and the possible conformal types for the underlying Riemann surface are quite limited. 
 
An immediate consequence of \eqref{sitau} is that if $\rho < 0$ then $Y$ has no zeros. In particular, if a compact orientable surface admits a nontrivial solution of \eqref{vlike} with $\rho < 0$ then it is a torus. Similarly, for a nontrivial solution $(h, Y)$, if $\rho \neq 0$, then $\fd$ does not vanish at the zeros of $Y$, for if $Y$ and $\fd_{h}$ both vanish at some $p \in M$, then by \eqref{sitau}, $\rho = 0$. If $\varep = -1$ and $\tau > 0$ then by \eqref{sitau}, $\rho > 0$. 

\begin{lemma}\label{spneglemma}
Let $(h, Y)$ be a nontrivial solution of the real vortex equations on $M$. If $\sp < 0$ then $\varep = 1$, $\rho > 0$, and $|Y|^{2}_{h}$ has neither a maximum nor a positive minimum on $M$. In particular, $M$ is noncompact.
\end{lemma}
\begin{proof}
If $\sR_{h}^{2} - \varep \fd_{h}^{2} = \sp < 0$ then $\varep = 1$, and so $\rho = (\tau^{2} - \sp)/4 > 0$.  Since $\tau = \sR_{h} + 4|Y|_{h}^{2}$ is constant, a critical point of $|Y|^{2}_{h}$ is a critical point of $\sR_{h}$, and by \eqref{drd} such a point is either a zero of $\fd_{h}$ or a zero of $Y$. That $|Y|^{2}_{h}$ have either a maximum or a positive minimum at $p \in M$ yields the contradiction $0 > \sp = \sR_{h}(p)^{2}$. 
\end{proof}

\begin{lemma}\label{twozerolemma}
Let $(h, Y)$ be a nontrivial solution of the real vortex equations on a connected orientable surface $M$ and let $J$ be the complex structure determined by $h$ and the given orientation. If $Y$ is complete, then $(M, J)$ is biholomorphic to one of the following Riemann surfaces: the sphere $\proj^{1}(\com)$, the plane $\com$, the punctured plane $\com \setminus \{0\}$, a torus, the disc $\disc = \{z \in \com:|z| < 1\}$, the punctured disc $\disc \setminus \{0\}$, or an annulus $\annulus(r) = \{z \in\com: r < |z| < 1\}$. In particular  $M$ has abelian fundamental group, and $Y$ has no more than two zeros. Moreover:
\begin{enumerate}
\item\label{spherezero} If $Y$ has a zero and $M$ is compact, then $Y$ has two zeros, $M$ is a sphere, $\sp > 0$, and $\rho \geq 0$. Moreover, if $\rho = 0$ then $\varep = -1$.
\item\label{toruszero} If $Y$ has no zero and $M$ is compact, then $M$ is a torus. If $\varep = 1$ then $h$ is flat, $Y$ is parallel, and $\sp = 0$, while if $\varep = -1$ then $\rho < 0$.
\item If $Y$ has a zero and $M$ is noncompact, then $Y$ has one zero, $M$ is biholomorphic to $\com$ or $\disc$, and $\rho \geq 0$.
\item If $Y$ has no zeros and is not parallel then $M$ is noncompact.
\end{enumerate}
\end{lemma}
\begin{proof}
Since $Y$ is complete, the flow of $Y^{(1,0)}$ is a one-parameter group of biholomorphisms of $M$, so the biholomorphism group of $M$ is not discrete. The surfaces listed in the statement of the lemma are exactly those Riemann surfaces with nondiscrete automorphism group; see, e.g. Theorem V.$4.1$ of \cite{Farkas-Kra}. If compact, $M$ is a sphere or torus, and if, moreover, $Y$ has zeros, $M$ must be a sphere. In this case \eqref{sitau} implies $\rho \geq 0$ and Lemmas \ref{ricsollemma} and \ref{spneglemma} imply $\sp > 0$. Suppose $|Y|^{2}$ assumes its maximum at $p \in M$. By \eqref{dyh}, $\fd_{h}(p) = 0$, so if $\rho = 0$ there holds $(\sR_{h}(p) + 4\varep |Y|^{2}_{h}(p))^{2} = \tau^{2} = \sp =  \sR_{h}^{2}(p)$, which forces $\varep = -1$. The zeros of $Y$ are the critical points of $\mom$. Since $Y$ is not trivial, $\mom$ is not constant, so must have at least two critical points. This shows $Y$ has two zeros. This shows \eqref{spherezero}. If $M$ is a torus then $|Y|^{2}_{h}$ assumes a minimum at some $p \in M$. If $\varep = 1$, then, because $\tau = \sR_{h} + 4|Y|^{2}_{h}$, at such a point $\sR_{h}$ assumes a maximum, and so $\max_{M}\sR_{h} = \sR_{h}(p) = -\lap_{h}\log|Y|^{2}_{h}(p) \leq 0$. By the Gauss-Bonnet theorem this forces $h$ to be flat, and so $\log |Y|^{2}_{h}$ is harmonic, and hence constant. Hence $\tau = 4|Y|^{2}_{h}$ and $0 = 2d|Y|^{2}_{h} = \fd_{h}\star \ga$, so $\fd_{h} = 0$ and $Y$ is parallel. In \eqref{sitau} this implies $4\rho = \tau^{2}$, so $\sp = 0$. On the other hand, if $\varep = -1$ then $\sR_{h}$ has a minimum at $p$ and so $\min_{M}\sR_{h} = \sR_{h}(p) = -\lap_{h}\log|Y|^{2}_{h}(p) \leq 0$, and by Gauss-Bonnet the inequality must be strict. Sinces $Y$ does not vanish at $p$ it follows from \eqref{dfd} that $\fd_{h}(p) = 0$ and so $\sp = \sR_{h}^{2}(p)$. On the other hand, $\tau  = \sR_{h}(p) - 4|Y|_{h}^{2}(p) < \sR_{h}(p) < 0$, so $\tau^{2} > \sR_{h}^{2}(p) = \sp$. Hence $4\rho = \sp - \tau^{2} <0$. This shows \eqref{toruszero}. If $M$ is noncompact and $Y$ has a zero, then both $M$ and the complement $M^{\ast}$ of the zeros of $Y$ must be among the surfaces listed in the statement of the lemma. The only possible pairs are $M = \com$ and $M^{\ast} = \com\setminus\{0\}$ and $M = \disc$ and $M^{\ast} = \disc \setminus \{0\}$. In this case $\rho \geq 0$ by \eqref{sitau}. Finally if $Y$ has no zeros and is not parallel, then by the preceeding, $M$ cannot be compact.
\end{proof}

\section{Ricci flows solving the real vortex equations}
\label{metricssection}
This section is dedicated to showing that there exist Ricci flows $\hh$ such that $(\hh, Y)$ solves the real vortex equations, and to constructing them explicitly. 

\subsection{}\label{uniquenesssection}
In any context in which its solutions are uniquene, e.g. on a complete manifold with bounded curvature, the Ricci flow preserves isometries in the sense that any isometry of the initial metric is an isometry of metrics later in the flow; see Corollary $1.2$ of \cite{Chen-Zhu}. In particular, a Killing field for the initial metric will be a Killing field all along the flow. What is not obvious is that the Ricci flow also preserves the compatibility condition \eqref{realvortex} between the metric and the Killing field.
Theorem \ref{flowtheorem} shows that given a solution $(h, Y)$ of the real vortex solutions there is locally a unique Ricci flow $\hh$ through $h$ such that $(\hh, Y)$ solves the real vortex equations. First, there is proved Lemma \ref{uniquenesslemma} which shows that if $\hh$ is a Ricci flow such that $(\hh, Y)$ solves \eqref{vlike}, then $\rho$ is constant along the flow, the sign of $\sp$ is preserved along the flow, and $\tau$ is monotonic along the flow. In addition to helping organize the possible solutions, these observations provide a priori restrictions on the values of the various parameters which are instrumental in the proof of Theorem \ref{flowtheorem}. The conclusion of Lemma \ref{uniquenesslemma} can  be viewed as generalizing the conclusion of Lemma \ref{ricsollemma}. 

\begin{lemma}\label{uniquenesslemma}
Let $(M, J)$ be a Riemann surface, let $\hh$ be a Ricci flow representing the given conformal structure and depending smoothly on $t$ in some open interval $I \subset \rea$, and let $Y$ be a fixed vector field on $M$. Suppose that for all $t \in I$ the pair $(\hh, Y)$ solves the real vortex equations \eqref{vlike} with constant $\tau(t)$ and parameter $\varep \in \{\pm 1\}$. Let $\tau(t)$ and $\sp(t)$ be the constants \eqref{realvortex} and \eqref{spdefined} determined by $(\hh, Y)$. Then  $\rho = (\tau(t)^{2} - \sp(t))/4\varep$ is constant in $t$ and $\tau(t)$ and $\sp(t)$ solve 
\begin{align}\label{tseq}
&\tfrac{d}{dt}\tau = \tau^{2} - 4\varep\rho = \sp, && \tfrac{d}{dt}\sp = 2\tau\sp. 
\end{align}
In particular, either $\sp(t) = 0$ and $\tau(t)$ is constant for all $t \in I$, or $\sp(t)$ has a definite sign on $I$ and $\tau(t)$ is monotone on $I$. 
\end{lemma}

\begin{proof}
Along $\hh$ there hold $\tfrac{d}{dt}\sR_{\hh} = \lap_{\hh}\sR_{\hh} + \sR_{\hh}^{2}$ (see \cite{Hamilton-riccisurfaces}) and $\tfrac{d}{dt}|Y|^{2}_{\hh} = -\sR_{h}|Y|_{\hh}^{2}$, so by \eqref{ub5} of Lemma \ref{squaresconstantlemma},
\begin{align}
\begin{split}
\tfrac{d}{dt}\tau(t) & = \tfrac{d}{dt}(\sR_{\hh} + 4\varep|Y|^{2}_{\hh}) =  \lap_{\hh}\sR_{\hh} + \sR_{\hh}^{2} - 4\varep\sR_{\hh}|Y|^{2}_{\hh} = \sp(t).
\end{split}
\end{align}
Let $\ga(t) = \imt(Y)\hh$. Using \eqref{drd} yields
\begin{align}
\begin{split}
\tfrac{d}{dt}d\ga(t) &= - d(\sR_{\hh}\ga(t)) = -d\sR_{\hh}\wedge \ga(t) -\sR_{\hh}d\ga(t)\\  
& = 2\varep \fd_{\hh}\star \ga(t)\wedge \ga(t) + (1/2)\sR_{h}\fd_{h}\om_{h} = \left(-2\varep\fd_{\hh}|Y|^{2}_{\hh}  + (1/2)\sR_{h}\fd_{h}\right)\om_{h}.
\end{split}
\end{align}
Hence
\begin{align}
\begin{split}
(\tfrac{d}{dt}\fd_{\hh})\om_{\hh} & = \tfrac{d}{dt}(\fd_{h}\om_{h}) - \fd_{h}\tfrac{d}{dt}\om_{h} = -2\tfrac{d}{dt}d\ga(t) + \fd_{h}\sR_{h}\om_{h} = 4\varep\fd_{\hh}|Y|^{2}_{\hh}\om_{\hh},
\end{split}
\end{align}
showing that $\tfrac{d}{dt}\fd_{\hh} =  4\varep\fd_{\hh}|Y|^{2}_{\hh}$. Differentiating \eqref{sitau} along $\hh$ yields
\begin{align}\label{rhoconstant}
\begin{split}
\tfrac{d}{dt}\rho(t) & = 2|Y|^{2}_{\hh}\left(\varep\fd_{\hh}^{2} + \sp(t) - \tau(t)\sR_{\hh} + 4\varep \sR_{\hh}|Y|^{2}_{\hh}\right) = 0.
\end{split}
\end{align}
By \eqref{rhoconstant}, $\tfrac{d}{dt}\sp = \tfrac{d}{dt}(\tau^{2}) = 2\tau\sp$. Hence, $\sp(t) = \sp(t_{0})\exp\left\{2\int_{t_{0}}^{t}\tau(x)\,dx\right\}$ for any $t_{0} \in I$, from which the claim about the sign of $\sp(t)$ is apparent. The monotonicity of $\tau$ then follows from $\tfrac{d}{dt}\tau = \sp$.
\end{proof}

By Lemma \ref{uniquenesslemma}, if $(\hh, Y)$ solves \eqref{vlike} and $\sp(t)$ is nonzero for some $t$, then $\sp(t)$ is never zero. In this case $\tau(t)$ and $\sp(t)$ solve the equations \eqref{tseq}. Explicit expressions for $\tau(t)$ and $\sp(t)$ are most easily found by observing that $m(t) = |\sp(t)|^{-1/2}$ solves $\ddot{m} = 4\varep\rho m$ with initial conditions $m(t_{0}) = |\sp(t_{0})|^{-1/2}$ and $\dot{m}(t_{0}) = - \tau(t_{0})|\sp(t_{0})|^{-1/2}$. Letting $\la = 1$ or $\j$ as $\varep$ is $1$ or $-1$, and letting $\sp_{0} = \sp(t_{0}) \neq 0$ and $\tau_{0} = \tau(t_{0})$, there result:
\begin{align}\label{sptau}
\begin{split}
\sp(t)& = \begin{cases} 4\sp_{0}\varep\rho \left(2\la\sqrt{\rho}\cosh 2\la \sqrt{\rho}(t - t_{0}) - \tau_{0}\sinh 2\la \sqrt{\rho}(t - t_{0})\right)^{-2} & \,\,\text{if}\,\, \rho \neq 0,\\
\sp_{0}(1 - \tau_{0}(t - t_{0}))^{-2} & \,\,\text{if}\,\, \rho = 0.
\end{cases}\\
\tau(t) &= \begin{cases} 
2\la\sqrt{\rho}\left(\frac{\tau_{0}\cosh 2\la \sqrt{\rho}(t - t_{0}) - 2\la \sqrt{\rho}\sinh2\la \sqrt{\rho}(t - t_{0})}{2\la \sqrt{\rho}\cosh 2\la \sqrt{\rho}(t - t_{0}) - \tau_{0}\sinh 2\la \sqrt{\rho}(t - t_{0})} \right)
& \,\,\text{if}\,\, \rho \neq 0,\\
\tau_{0}\left(1 - \tau_{0}(t - t_{0})\right)^{-1}  & \,\,\text{if}\,\, \rho = 0.
\end{cases}
\end{split}
\end{align}
The expressions in the $\rho = 0$ case are the $\rho \to 0$ limits of the $\rho \neq 0$ expressions. The choice of the numerical value of $t_{0}$ is arbitrary, and can be made so that the expressions assume more convenient forms. For instance, when $\tau_{0} \neq 0$, taking $t_{0} = -\tau_{0}^{-1}$, the expressions in the $\rho = 0$ case become $\tau(t) = -t^{-1}$ and $\sp(t) = t^{-2}$. When $\rho \neq 0$ the expressions in \eqref{sptau} can be simplified considerably by an appropriate choice of $t_{0}$ depending on the values of $\rho$, $\varep$, and $\sp_{0}$. Precisely, by an appropriate choice of $t_{0}$, $\tau(t)$ and $\sp(t)$ can be assumed to have the forms:
\begin{align}\label{taut1}
&\tau(t) = -2\la \sqrt{\rho} \coth 2\la\sqrt{\rho}t,& &\sp(t) = 4\varep \rho \csch^{2}(2\la \sqrt{\rho}t),& &\text{when}\,\, \varep\rho > 0, \sp > 0 \\
\label{taut1b}
&\tau(t) = -2 \sqrt{|\rho|} \cot 2\sqrt{|\rho|}t,& &\sp(t) = 4|\rho| \csc^{2}(2\sqrt{|\rho|}t),&  &\text{when}\,\,  \varep\rho < 0, \sp > 0\\
\label{tautneg}
&\tau(t) = -2\sqrt{\rho} \tanh 2\sqrt{\rho}t,& &\sp(t) = -4\rho \sech^{2}(2\sqrt{\rho}t),&  &\text{when}\,\, \rho > 0, \sp < 0.
\end{align}
In \eqref{tautneg} the parameters $\varep$ and $\la$ are omitted because when $\sp < 0$ they must take the values $\varep = 1 = \la$. In the case $\varep\rho > 0$ and $\sp_{0} > 0$, \eqref{sptau} yields \eqref{taut1} upon letting $q$ be the unique real number such that $\cosh q = \tau_{0}\sp_{0}^{-1/2}$ and $\sinh q = 2\la \sqrt{\rho}\sp_{0}^{-1/2}$ and taking $t_{0} = -q/(2\la\sqrt{\rho})$. In the case $\varep\rho > 0$ and $\sp_{0} < 0$, \eqref{sptau} yields \eqref{tautneg} upon letting $q$ be the unique real number such that $\cosh q = 2\la\sqrt{\rho}|\sp_{0}|^{-1/2}$ and $\sinh q = \tau_{0}|\sp_{0}|^{-1/2}$ and taking $t_{0} = -q/(2\la\sqrt{\rho})$. In the case $\varep\rho < 0$ and $\sp_{0} > 0$, \eqref{sptau} yields \eqref{taut1b} upon letting $q$ be the unique real number such that $\cos q = \tau_{0}\sp_{0}^{-1/2}$ and $\sin q = 2\sqrt{|\rho|}\sp_{0}^{-1/2}$ and taking $t_{0} = -q/(2\sqrt{|\rho|})$.

 Although it is convenient to record both, the expressions \eqref{taut1} and \eqref{taut1b} are actually the same; when $\varep\rho < 0$, \eqref{taut1b} results from \eqref{taut1} using the identity $\sinh(ix) = i\sin(x)$.

\subsection{}\label{rotationconventionsection}
 Let $z = x + \j y = e^{s}e^{ir}$ be the standard coordinate on the Riemann sphere, where the coordinates $s \in \rea$ and $r \in [0,2\pi)$ on the punctured plane $\rea^{2}\setminus\{0\}$ are called \textit{cylindrical} because the metric $|z|^{-2}|dz|^{2} = dr^{2} + ds^{2}$ is the metric of a flat cylinder. For $\be > -1$, a metric $h$ on a surface $M$ is said to have a \textit{conical singularity} of angle $2\pi(\be + 1)$ at a point $p \in M$ if there is an open neighborhood $U \subset M$ of $p$ and a diffeomorphism mapping $U$ into $\com$ such that $p$ is mapped to $0$ and there is a smooth function $f$ such that the pullback to $U$ of the metric $e^{f}|z|^{2\be}|dz|^{2} = e^{f}e^{2(\be + 1)s}(dr^{2} + ds^{2})$ is equal to $h$ on $U \setminus \{p\}$. If the same condition is satisfied but with $\be = -1$, then $h$ is said to have a logarithmic singularity or \textit{cusp} at $p$. On $\sphere$ the metric 
\begin{align}\label{rhospheremetric}
\frac{2\rho|z|^{\sqrt{\rho} - 2}|dz|^{2}}{(1 + |z|^{\sqrt{\rho}})^{2}}=  \frac{\rho(dr^{2} + ds^{2})}{\cosh \sqrt{\rho}s + 1} =  \frac{\rho(dr^{2} + ds^{2})}{2\cosh^{2}(\sqrt{\rho}s/2)},
\end{align}
has scalar curvature $1$ and volume $4\pi\sqrt{\rho}$, with cone points of angle $\pi\sqrt{\rho}$ at the poles. On the two discs in $\sphere$ complementary to the circle $s = 0$ ($|z| = 1$) the metric
\begin{align}\label{rhohyperbolicmetric}
 \frac{2\rho|z|^{\sqrt{\rho} - 2}|dz|^{2}}{(1 - |z|^{\sqrt{\rho}})^{2}} = \frac{\rho(dr^{2} + ds^{2})}{\cosh \sqrt{\rho}s - 1} = \frac{\rho(dr^{2} + ds^{2})}{2\sinh^{2} (\sqrt{\rho}s/2) },
\end{align}
has constant scalar curvature $-1$, with cone points of angle $\pi\sqrt{\rho}$ at the centers of the disks.
The metrics \eqref{rhospheremetric} and \eqref{rhohyperbolicmetric} have in the complementary chart with coordinates $\tilde{z} = -1/z = -e^{-s}e^{\j r}$ the same expressions, with $z$ replaced by $\tilde{z}$.

Suppose $M$ is an orientable surface equipped with a Riemannian metric $h$, a compatible complex structure $J$, and K\"ahler form $\om_{h}$, and let $Y$ be a nontrivial Killing field. Let $\tilde{M}$ be the universal cover of $M$ and denote the pullbacks to $\tilde{M}$ of objects on $M$ in the same way as the objects themselves. 
Let $M^{\ast}$ be the open dense subset of $M$ on which $Y$ is nonvanishing. Let $\tilde{M}^{\ast}$ be the universal cover of $M^{\ast}$. Let $\mom$ be the moment map on $\tilde{M}$ defined by $d\mom = -\star \ga$. Define $u = |Y|^{2}_{h}$ and $w = u^{-1}$. Since $d(w\ga) = dw \wedge \ga - wF = 0$ on $\tilde{M}^{\ast}$, there is an $r \in \cinf(\tilde{M}^{\ast})$ such that $dr = w\ga$. Since $dw \wedge \star \ga = 0$ there holds $d(w \star \ga) = 0$ and so there is an $s \in \cinf(\tilde{M}^{\ast})$ such that $ds = -w\star \ga$. Then $\ga = udr$ and $-\star \ga = d\mom = uds$. Also, $\delbar(s + \j r) = 0$, so $s + \j r$ and $z = x + \j y = e^{s}e^{\j r}$ are holomorphic functions on $\tilde{M}^{\ast}$. By construction $h =  udr^{2} + w d\mom^{2} = u(dr^{2} + ds^{2})$, and the associated K\"ahler form is $\om_{h} = d\mom \wedge dr = uds \wedge dr$. Let $\phi_{a}$ and $\psi_{b}$ be local flows of $Y$ and $-JY$. Since $\lie_{Y}J = 0$, there holds $[Y, JY] = 0$, and so these flows commute. For any $p \in M^{\ast}$, $\tfrac{d}{da}(r \circ \phi_{a}(p)) = dr(Y_{\phi_{a}(p)}) = 1$ and $\tfrac{d}{db}(s\circ \psi_{b}(p)) = -ds(JY) = 1$, so $r\circ \phi_{a}(p) - r(p) = a$ and $s\circ \psi_{b}(p) - s(p) = b$. Normalizing $r$ and $s$ so that $r(p) = 0 = s(p)$, this means that for every $p$ there is an open neighborhood on which $r$ and $s$ are local coordinates such that the origin corresponds to $p$, $Y = \pr_{r}$, and $\pr_{s} = -JY$. Such coordinates will be called \textit{cylindrical coordinates centered at $p$}. With respect to the coordinate $z = x + \j y = e^{s}e^{\j r}$, $J$ is the standard complex structure on $\com$, and $\pr_{r}= x\pr_{y} - y\pr_{x}$ and $\pr_{s} = x\pr_{x} + y\pr_{y}$. From the fact that $Y = \pr_{r}$ is Killing it follows that the partial derivative $u_{r}$ is zero and $u$ is locally constant on the level sets of $\mom$. 

The metric $h^{\ast} = wh$ on $M^{\ast}$ is flat. The scalar curvature of a metric $g = a h^{\ast}$ is $\sR_{g} = -a^{-1}\lap_{h^{\ast}}\log a$, where $\lap_{h^{\ast}}$ means the Laplacian of the flat metric $h^{\ast}$. Let $\pr$ be the Levi-Civita connection of $h^{\ast}$, and observe that $\{Y, -JY\}$ is a $\pr$-parallel frame. From $-2d\log w(JY) = \fd_{h}$ and \eqref{dfd} it follows that $\sR_{h} = -u^{-1}(\pr d\log u)(JY, JY) = w(\pr d\log w)(JY, JY) = \sR_{h}$. Use subscripts to indicate covariant derivatives with respect to $\pr$. In coordinates, $\sR_{h} = -u^{-1}(\log u)_{ss}$, and $F = -d\log u \wedge \ga = -(\log u)_{s}\om_{h}$, so $\fd_{h} = -2(\log u)_{s}$. 

\subsection{}\label{sol2section}
Suppose $(h, Y)$ solves the real vortex equations \eqref{vlike} with parameters $\tau$ and $\varep$. Work on the complement $M^{\ast}$. By \eqref{dyh}, $\fd_{h} = 2d\log u(JY) = -2d\log w(JY)$. In \eqref{sitau} this yields 
\begin{align}\label{sitau2}
-4\varep \rho = \sp - \tau^{2} = -4\varep(d\log u)(JY)^{2} - 8\varep\tau u + 16u^{2}.
\end{align}
In terms of $w = u^{-1}$, \eqref{sitau2} becomes
\begin{align}\label{weq2}
dw(JY)^{2} - 4\varep + 2\tau w -\rho w^{2} = 0.
\end{align}
Differentiating \eqref{weq} along $JY$ shows that
\begin{align}\label{wss2}
0 = 2dw(JY)\left((\pr dw)(JY, JY) - \rho w + \tau \right) = -w\fd_{h}\left((\pr dw)(JY, JY) - \rho w + \tau \right).
\end{align}
By \eqref{fd1} of Lemma \ref{squaresconstantlemma}, if $Y$ is not parallel then the zero set of $\fd_{h}$ in $M^{\ast}$ is a union of smoothly immersed curves, so by continuity \eqref{wss2} implies that
\begin{align}\label{wss3}
(\pr dw)(JY, JY) = \rho w - \tau,
\end{align}
on $M^{\ast}$. Let $r$ and $s$ be local cylindrical coordinates centered on $p \in M^{\ast}$. Equations \eqref{weq2} and \eqref{wss3} become
\begin{align}\label{weq}
0 & = w_{s}^{2} - 4\varep + 2\tau w -\rho w^{2},\\
\label{wss}
w_{ss} & = \rho w - \tau.
\end{align}
While the general solution of \eqref{wss} has two free parameters, \eqref{weq} imposes on them a further relation, leaving a single degree of freedom. 
A consequence of \eqref{wss} that will be needed in the proof of Theorem \ref{flowtheorem} is that if $(h, Y)$ solves \eqref{vlike} then in a neighborhood of any point of $M^{\ast}$ the metric $h$ is real analytic. Consequently its curvature is also real analytic. In particular, if $h$ has constant curvature on an open subset of $M^{\ast}$ then it is flat on all of $M^{\ast}$.

Solutions $(h, Y)$ to \eqref{vlike} can be constructed by reversing the preceeding. Given $\si$ and $\tau$, one solves \eqref{weq} for $w$, defines $u = w^{-1}$, and defines $h = u(dr^{2} + ds^{2})$ and $Y = \pr_{r}$. Whether the resulting solution extends when $s \to \pm \infty$ has to be analyzed on a case by case basis.

\subsection{}
Suppose $\hh$ is a one-parameter family of metrics and $Y$ is a Killing field for each $\hh$. Write $\hh = u(t, p)h^{\ast} = w(t, p)^{-1}h^{\ast}$ where $h^{\ast} = |Y|^{-2}_{h}h$ and $p \in M^{\ast}$. That $\hh$ moreover evolve by the Ricci flow $\frac{d}{dt}h = -\sR_{h}h$ is equivalent to the equation
\begin{align}\label{lde}
w_{t} =  w(\pr dw)(JY, JY) - dw(JY)^{2}.
\end{align}
In local cylindrical coordinates $r$ and $s$, \eqref{lde} becomes
\begin{align}\label{lde2}
w_{t} = ww_{ss} - w_{s}^{2}.
\end{align}
Equation \eqref{lde2} is equivalent to $u$ solving the \textit{logarithmic diffusion equation} $u_{t} = (\log u)_{ss}$. 

Suppose that for each $t$ in some interval $I$ the metrics $\hh$ and the fixed vector field $Y$ together solve \eqref{vlike} with parameters $\varep$ and $\tau(t)$ on the oriented surface $M$. 
Note that the induced conformal structure does not depend on $t$. Combining \eqref{weq}, \eqref{wss}, and \eqref{lde2} shows that for $\hh$ also to be a solution to the Ricci flow necessitates
\begin{align}\label{wt}
w_{t} = w(\rho w - \tau) -  4\varep + 2\tau w -\rho w^{2} = \tau w - 4\varep,
\end{align}
in which $\tau$ is a function of $t$. 

\subsection{}
Theorem \ref{flowtheorem} shows that when $(h, Y)$ solves the real vortex equations \eqref{vlike} on $M$ then locally on $M^{\ast}$ there is a unique Ricci flow $\hh$ through $h$ such that $(\hh, Y)$ solves \eqref{vlike}. The uniqueness follows from the proof of Lemma \ref{uniquenesslemma}. The assumption that $(\hh, Y)$ solves both the Ricci flow and the real vortex equations yields the equations \eqref{wss} and \eqref{wt}. Theorem \ref{flowtheorem} shows that the solution of \eqref{wt} obtained with the initial data determined by a solution of \eqref{vlike} necessarily satisfies \eqref{wss}. This shows that given $(h, Y)$ solving \eqref{vlike} then locally there is a unique Ricci flow $\hh$ such that $(\hh, Y)$ solves \eqref{vlike}. A comparably general global statement is not feasible without assuming more, e.g. that the surface be compact. On the other hand, more detailed information  about the solutions to the real vortex equations than that provided by Theorem \ref{flowtheorem} can be obtained on a case by case basis by solving the equations \eqref{wss} and \eqref{wt} explicitly, and the latter part of this section is devoted to describing the resulting metrics in detail.

\begin{theorem}\label{flowtheorem}
Let $(M, J)$ be a Riemann surface and suppose that there are a metric $h$ representing the given conformal structure and a complete vector field $Y$ that together solve the real vortex equations \eqref{vlike} for a given constant $\tau_{0} \in \rea$ and given parameter $\varep \in\{\pm 1\}$. Define $\rho\in \rea$ by $4\varep \rho = \tau_{0}^{2} - \sp_{0}$, where $\sp_{0}$ is the constant \eqref{spdefined} determined by $(h, Y)$. Let $\tau(t)$ be the unique solution of $\tau_{t} = \tau^{2} - 4\varep \rho$ satisfying $\tau(t_{0}) = \tau_{0}$. For each $p \in M$ there are a relatively compact open neighborhood $U_{p} \subset M$ containing $p$, an interval $I \subset \rea$ containing $t_{0}$ and contained in the maximal domain of definition of $\tau(t)$, and a unique smooth Ricci flow $\hh$ defined for $(t, q) \in I \times U_{p}$ such that $h(t_{0})$ equals the restriction of $h$ to $U_{p}$ and such that for all $t \in I$, $(\hh, Y)$ solves \eqref{vlike} on $U_{p}$ with constant $\tau(t)$ and parameter $\varep$. For the spatial constant $\sp(t)$ associated to $(\hh, Y)$ as in \eqref{spdefined}, the expression $\tau(t)^{2} - \sp(t)$ is constant in $t$ for $t \in I$, equal to $4\varep \rho$.
\end{theorem}

\begin{proof}
The completeness of $Y$ is assumed so that none of what follows depends on the domain of definition of the flow of $Y$, and will not be mentioned again in the proof. If $Y$ is parallel then $h$ must be flat and there is nothing to show, so it can be assumed that $Y$ is not parallel. 
Let $M^{\ast}$ be the complement of the discrete set of zeros of $Y$. Define functions $u \in \cinf(M)$ and $w\in \cinf(M^{\ast})$ by $u = |Y|^{2}_{h}= w^{-1}$. Let $\tau(t)$ be the unique solution of $\tau_{t}  = \tau^{2} - 4\varep \rho$ satisfying $\tau(t_{0}) = \tau_{0}$. It is defined on some maximal connected open subset $\hat{I} \subset \rea$. Let $\sp(t)$ be the unique solution of $\sp_{t} = 2\tau\sp$ satisfying $\sp(t_{0}) = \sp_{0}$. Then $\tau(t)$ and $\sp(t)$ are as in \eqref{sptau} and $\sp(t) = \sp_{0}\exp\left\{2\int_{t_{0}}^{t}\tau(x)\,dx\right\}$. 
For $p \in M$ the unique solution $L(t, p)$ of the initial value problem
\begin{align}\label{fos}
\begin{split}
L_{t} & = \tau(t)L - 4\varep u(p), \qquad \qquad L(t_{0}, p) = 1,
\end{split}
\end{align}
is given explicitly by 
\begin{align}\label{lsol}
L(t, p) = \begin{cases}
\left(1 - \tfrac{\tau_{0}}{\rho}u(p)\right)\left(\tfrac{\sp(t)}{\sp_{0}}\right)^{1/2} + \tfrac{\tau(t)}{\rho}u(p) & \,\,\text{if}\,\, \rho \neq 0  \,\, \text{and}\,\, \sp_{0} \neq 0,\\
\tfrac{\tau(t)}{\rho}u(p) & \,\,\text{if}\,\, \rho \neq 0  \,\, \text{and}\,\, \sp_{0} = 0,\\
\left(1 - \tfrac{4\varep}{\tau_{0}}u(p)\right)\left(\tfrac{\sp(t)}{\sp_{0}}\right)^{1/2} + \tfrac{4\varep}{\tau_{0}}u(p)& \,\,\text{if}\,\, \rho =0 \,\, \text{and}\,\, \tau_{0} \neq 0,\\
1 + 4\varep(t_{0} - t)u(p) & \,\,\text{if}\,\, \rho =0 \,\, \text{and}\,\, \tau_{0} = 0.
\end{cases}
\end{align} 
Define $\hh_{p} = L(t, p)^{-1}h_{p}$. For each $p \in M$ there is some maximal connected open subset $\hat{I}_{p} \subset \rea$ containing $t_{0}$ and such that $L(t, p)$ is defined and positive for all $t \in \hat{I}_{p}$. Because the solution of \eqref{fos} depends smoothly on the initial data, there is a relatively compact neighborhood $U_{p} \subset M$ of $p$ and a maximal connected open interval $I$ containing $t_{0}$ such that the metric $\hh$ is defined on $U_{p}$ for all $t \in I$. 

Differentiating $dL(Y)$ in $t$ gives $\tfrac{d}{dt}dL(Y) = \tau dL(Y)$. Since $dL(Y)_{t = t_{0}} = 0$, this implies $dL(Y) = 0$ at $p$ for all $t$ for which $L(t, p)$ is defined. Consequently, the vector field $Y$ is Killing for $\hh$ where $\hh$ is defined. The function $W(t, p) = L(t, p)w(p)$ defined for $p \in M^{\ast}$ is smooth and solves the initial value problem
\begin{align}\label{fos2}
\begin{split}
W_{t} & = \tau(t)W - 4\varep, \qquad \qquad W(t_{0}, p) = w(p).
\end{split}
\end{align}
The function $V(t, p) = dW(JY)^{2} - \rho W^{2} + 2\tau W - 4\varep$ solves
\begin{align}\label{weqt}
\begin{split}
V_{t} &= 2dW(JY)dW_{t}(JY) - 2\rho WW_{t} + 2\tau W_{t} + 2(\tau^{2} - 4\varep\rho)W,\\
& = 2\tau dW(JY)^{2} + 2(\tau - \rho W)(\tau W - 4\varep) + 2\tau^{2}W - 8\varep \rho W = 2\tau V,
\end{split}
\end{align}
for $t \in I_{p}$. By \eqref{weq}, $V(t_{0}, p) = 0$, and so \eqref{weqt} implies $V(t, p) = 0$ for all $p \in M^{\ast}$. Let $\pr$ be the Levi-Civita connection of the flat metric $h^{\ast}$ on $M^{\ast}$. 
Differentiating $0 = dW(JY)^{2} - \rho W^{2} + 2\tau W - 4\varep$ along $JY$ yields
\begin{align}\label{weqt2}
\begin{split}
0 & = 2dW(JY)\left((\pr dW)(JY, JY) - \rho W + \tau\right).
\end{split}
\end{align}
By \eqref{lsol}, $dW$ is a multiple of $dw$, and so, by \eqref{dyh}, $dW(JY)$ is a nonzero multiple of $w\fd_{h}$. Because $Y$ is not $h$-parallel, by \eqref{fd1} of Lemma \ref{squaresconstantlemma}, the zero set of $\fd_{h}$ in the complement of its zero set is a union of immersed curves, so $dW(JY)$ is nonzero off these curves, and \eqref{weqt2} implies 
\begin{align}\label{weqt3}
\begin{split}
(\pr dW)(JY, JY) = \rho W - \tau,
\end{split}
\end{align}
which holds on all of $M^{\ast}$. Then
\begin{align}
\begin{split}
\sR_{\hh} + 4\varep|Y|^{2}_{\hh} 
& = -W\lap_{h^{\ast}}\log W + 4\varep W^{-1} = W(\pr d\log W)(JY, JY) + 4\varep W^{-1} \\
& = (\pr d W)(JY, JY) - (d\log W)(JY)^{2} + 4\varep W^{-1}\\
& = \rho W - \tau  - W^{-1}(\rho W^{2} - 2\tau W) = \tau(t),
\end{split}
\end{align}
so that $(\hh, Y)$ solves \eqref{vlike} on $M^{\ast}$ with parameters $\varep$ and $\tau(t)$. Hence $\tau(t) = \sR_{\hh} + 4\varep|Y|^{2}_{\hh}$ on $M^{\ast}$, and by continuity of $\sR_{\hh}$ and $|Y|^{2}_{\hh}$, the same identity holds at $q \in M$, showing that, when defined, $(\hh, Y)$ solves \eqref{vlike} with parameters $\tau(t)$ and $\varep$.

From \eqref{weqt3} there results
\begin{align}
\begin{split}
W_{t} &- W(\pr dW)(JY, JY) + dW(JY)^{2} = \tau W - 4\varep - W(\rho W - \tau) + dW(JY)^{2} = 0.
\end{split}
\end{align}
By \eqref{lde2} this shows that the metric $\hh$ solves the Ricci flow on $M^{\ast}$. Now suppose $q \in M$ is a zero of $Y$. Since $v(t) = (\sp(t)/\sp_{0})^{1/2}$ solves $v_{t} = \tau(t)v$, it solves \eqref{lsol} for $p = q$, and so by the uniqueness of solutions to \eqref{lsol} there holds $L(t, q) = (\sp(t)/\sp_{0})^{1/2}$. Hence $\tfrac{d}{dt}\hh_{q} = \tfrac{d}{dt}( (\sp(t)/\sp_{0})^{-1/2}h_{q}) = -\tau(t)\hh_{q}$.  Since $\sR_{\hh}(q) = \tau(t)$, this shows that $\hh$ solves the Ricci flow at $q$. 
By Lemma \ref{uniquenesslemma}, the parameter $\rho(t)$ defined by $4\varep \rho(t) = \tau(t)^{2} - \sp(t)$ is constant, equal to $\rho$.

The preceeding proves that for every $p \in M$ there is a relatively compact open neighborhood $U\subset M$ of $p$ and an open connected neighborhood $I \subset \rea$ of $t_{0}$ such that for all $t \in I$ the metric $\hh$ solves the Ricci flow and the pair $(\hh, Y)$ solves \eqref{vlike} with parameters $\tau(t)$ and $\varep$. Suppose $g(t)$ is another Ricci flow such that $(g(t), Y)$ solves \eqref{vlike} on $U$. Since $g(t)$ remains within the conformal class of the initial metric $h$, as in the proof of Lemma \ref{uniquenesslemma}, for each $t$ the corresponding function $w(t, s)$ on $U\cap M^{\ast}$ solves \eqref{wt}, or, equivalently, \eqref{fos2}. Since the solution of \eqref{fos2} is unique, this shows that $g(t) = \hh$ on $U \cap M^{\ast}$, and so on $U$, by continuity.
\end{proof}

Although without imposing some conditions on the geometry of $M$ it is complicated to say anything more precise than Theorem \ref{flowtheorem}, the Ricci flow constructed in Theorem \ref{flowtheorem} is in some sense defined on all of $M^{\ast}$. The qualification is that the maximal domain of definition for $t$ can in principle depend on $p \in M^{\ast}$. 

\begin{corollary}\label{flowcorollary}
If $(h, Y)$ solves \eqref{vlike} on the compact orientable surface $M$, then there is a unique Ricci flow $\hh$ defined for all $t$ in some open interval $I$ containing $t_{0}$ such that $h(t_{0}) = h$, and the pair $(\hh, Y)$ solves \eqref{vlike} for all $t \in I$. 
\end{corollary}
\begin{proof}
The short time existence and uniqueness of the Ricci flow under the stated conditions is well known. By Theorem \ref{flowtheorem}, $M$ can be covered by relatively compact open neighborhoods on each of which there is a unique Ricci flow through $h$ forming with $Y$ a solution to \eqref{vlike}. Passing to a finite subcover, the resulting solutions patch together to give a globally defined Ricci flow defined on some interval and which must therefore be the unique Ricci flow through $h$. 
\end{proof}

In the rest of the paper there are analyzed the possible solutions of \eqref{wss} and the resulting Ricci flows. Let $\la$ be $1$ or $\j$ as $\varep$ is $1$ or $-1$, so that $\varep = \la^{2}$. The solutions are constructed on a case by case basis, depending on the values of the parameters $\sp$, $\rho$, and $\la$. The explicit solutions yield somewhat more precise information than what is given by Theorem \ref{flowtheorem}, in particular in relation to what happens at the zeros of $Y$. Additionally, there will be constructed solutions $(\hh, Y)$ of \eqref{vlike} in which the Ricci flow $\hh$ has conical singularities at the zeros of $Y$.
 
\subsubsection{}
Let $(M, J)$ be one of the surfaces in Lemma \ref{twozerolemma} with its standard conformal structure and let $Y$ be a fixed vector field generating an action on $M$ by biholomorphisms. On the preimage of $M^{\ast}$ in the universal cover of $M$, it is always possible to choose global coordinates $z = e^{s+\j r}$ as in section \ref{rotationconventionsection} with respect to which $Y = \pr_{r}$ and the given conformal structure is represented by the flat metric $h^{\ast} = dr^{2} + ds^{2}$. Note that $s$ is determined up to translation, and that changing the sign of $s$ corresponds to replacing $Y$ by $-Y$ (equivalently, replacing $J$ by $-J$). Suppose that $\hh$ is a Ricci flow, defined for each $p \in M^{\ast}$ on some definite interval $I_{p}$, such that $(\hh, Y)$ solves \eqref{vlike} for some $\varep$, $\rho$, and $\sp(t)$, where $\sp(t)$ (and $\tau(t)$) are as in \eqref{sptau}. Then $\hh = u(t, p)h^{\ast}$ where $u(t, p) = |Y|_{\hh}^{2}$, and $w(t, p) = u(t, p)^{-1}$ solves \eqref{weq}, \eqref{wss}, \eqref{lde2}, and \eqref{wt}. 

Since $w$ solves $w_{ss} = \rho w - \tau(t)$, it can be written as the sum of a particular solution and a solution of the homogeneous equation $f_{ss} = \rho f$. The particular solution can be taken to be $\tau(t)/\rho$ if $\rho \neq 0$, where $\tau$ is as in \eqref{taut1}-\eqref{tautneg}, and $-2^{-1}\tau(t)s^{2}$ if $\rho = 0$. Since $w(t, p) = f(t) + V(t, p)$ must solve \eqref{wt}, there holds $V_{t} = \tau V$. Either $\sp(t)$ is identically zero or it is never zero. In the second case, $V(t, p) = |\sp(t)|^{1/2}v(p)$ for some smooth function $v(p)$ such that $dv(Y) = 0$. Then, since $w(t, p) = f(t) + |\sp(t)|^{1/2}v(p)$ must solve $w_{s}^{2} - 4\varep + 2\tau(t)w - \rho w^{2} = 0$, $v$ solves $v_{s}^{2} - \rho v^{2} = -\sign(\sp)/\rho$ and so also $v_{ss} = \rho v$. In the following sections the metrics resulting from the various possible choices of $f$ and $v$, their domains of definition, and their geometric properties are obtained by specializing the formulas just obtained.

\subsubsection{}\label{scalingsection}
Since $e^{-c}h(e^{c}t)$ solves the Ricci flow if $h(t)$ does, if $\hh$ is a Ricci flow such that $(\hh, Y)$ solves \eqref{vlike}, then $(e^{-c}h(e^{c}t), e^{c}Y)$ also solves \eqref{vlike} for any $c \in \rea$. Thus the one-parameter families of solutions of \eqref{vlike} originating in scaling equivalent solutions of \eqref{vlike} are equivalent modulo scaling and an internal scaling reparameterization of the Ricci flow. 
However, in constructing solutions to \eqref{vlike}, some care is necessary when considering rescalings. Replacing $(h, Y)$ by $(\bar{h}, \bar{Y}) = e^{c/2}\cdot(h, Y) = (e^{c}h, e^{-c}Y)$ replaces $\tau$ and $\sp$ by $e^{-c}\tau$ and $e^{-2c}\sp$, so replaces $\rho$ by $e^{-2c}\rho$. The coordinates $r$ and $s$ are determined by the flows of $JY$ and $Y$, and so upon rescaling are replaced by the parameters $\bar{r} = e^{c}r$ and $\bar{s} = e^{c}s$ corresponding to $e^{-c}JY$ and $e^{-c}Y$. It is straightforward to check that if $\bar{w}(\bar{s})$ is the function obtained from $(\bar{h}, \bar{Y})$ as $w$ was obtained from $(h, Y)$, then $\bar{w}(s) = e^{c}w(e^{-c}s)$, so that $\bar{w}$ solves \eqref{weq} with $\bar{\tau}$ and $\bar{\sp}$ in place of $\tau$ and $\sp$ if and only if $w$ solves \eqref{weq}. In this sense, by rescaling $(h, Y)$ the parameter $\rho$ can be normalized to take a given value, e.g. $0$, $4$, or $-4$. Precisely, if $(h, Y)$ is given, it determines a $\rho$, and there is a scaling equivalent pair $(\bar{h}, \bar{Y}) = e^{c/2}\cdot(h, Y)$ determining $e^{c}\rho$. However, such a rescaling presupposes that the scaling equivalence class of $(h, Y)$ is known a priori. If, instead, \eqref{weq} is to be solved in order to construct $(h, Y)$ by inverting the procedure used to derive \eqref{weq}, then the particular value of $\rho$ may matter because the parameter $s$ has implicitly been fixed up to translations. For instance, if the metric resulting from the solution of \eqref{weq} is to be extended to some larger manifold there has to be analyzed whether a rescaling of $\rho$ can be achieved via a geometric or scaling equivalence of the resulting structure on this larger manifold. This need not be the case. There are scaling equivalent solutions to \eqref{vlike} on the punctured disk or punctured sphere which extend to the disk or the sphere with conical singularities at the punctures, but for which the extended solutions are no longer scaling equivalent; see the penultimate paragraph of section \ref{cigarsection} for an example.

\subsection{The soliton case: \texorpdfstring{$\sp = 0$}{sigma = 0} }\label{solitoncasesection}
The first case considered is $\sp = 0$. By Lemma \ref{ricsollemma}, the solutions of \eqref{vlike} with $\varep = -1$ have $h$ flat and $Y$ parallel. For this reason it will be assumed that $\varep = 1$; by Lemma \ref{ricsollemma}, the solutions of \eqref{vlike} with $\varep = 1$ are steady gradient Ricci solitons. In this case, since $\tau^{2} = 4\varep\rho = 4\rho$, either $\tau = 0 = \rho$, or $\rho > 0$ and $\tau = \pm 2\sqrt{\rho}$. 

\subsubsection{Case \texorpdfstring{$\sp = \tau = \rho = 0$}{sigma = tau = rho = 0}}
Consider the case $\tau = 0 = \rho$. By \eqref{weq}, $w_{s}^{2} = 4$, and so there is a function $b(t)$ such that $w = \pm 2s + b$. By \eqref{wt} there holds $-4 = w_{t} = b_{t}$, so, after a translation in $s$, $w$ may be supposed to have the form $w_{\pm} = \pm 2s - 4t$. The metrics 
\begin{align}\label{allzero}
\hhpm = \tfrac{dr^{2} + ds^{2}}{\pm 2s - 4t}  =  \pm \tfrac{|dz|^{2}}{2|z|^{2}(\log\left(e^{\mp 2t}|z|\right))}
\end{align}
are Ricci solitons, defined respectively on the half cylinders $\pm s > 2t$ and related by $-h_{+}(-t) = \hhm$. Via time dependent translations in $s$ they are diffeomorphic images of the fixed metrics $g_{\pm} = h_{\pm}(0)$. The metrics $g_{+} = - g_{-}$ and $g_{-}$ are defined, respectively, on the complement $\com \setminus \disc$ of the unit disk, and the punctured disk $\disc \setminus \{0\}$. The metric $g_{-}$ blows up as $z \to 0$ or $|z| \to 1$, so extends smoothly to no larger domain, as asserted above. By Lemma \ref{ricsollemma}, $g_{-}$ is a steady gradient Ricci soliton with potential equal to twice the moment map $\mom = \tfrac{1}{2}\log |s| = \tfrac{1}{2}\log|\log|z||$. The solution $(g_{+}, Y)$ of \eqref{vlike} on $\com \setminus \disc$ corresponds under the map $z \to 1/z$ to the solution $(g_{-}, -Y)$ of \eqref{vlike} on $\disc \setminus\{0\}$ obtained by replacing $Y$ by $-Y$. The curvature $\sR_{g_{-}} = \fd_{g_{-}} = 2s^{-1}$ is negative on the entire domain of $g_{-}$, but is not bounded from below; since $Y$ has no zero, this conclusion also follows from \eqref{ub1} of Lemma \ref{squaresconstantlemma}. Under the change of variables $q = (-s)^{1/2}$, $g_{-}$ becomes $2(dq^{2} + 4^{-1}q^{-2}dr^{2})$. In this form (modulo notation) this metric appears on p.$15$ of section $1.3.3$ of \cite{ricciflowii}. An integral curve $(r(x), s(x))$ of the geodesic vector field $U = -|Y|^{2}_{g_{-}}JY = (-2s)^{1/2}\pr_{s}$ such that $s(0) = s_{0}$ is given by $r(x) = r(0)$ and $s(x) = -(\sqrt{-2s_{0}}- x)^{2}/2$. Since this exists only for $x \in (-\infty, \sqrt{-2s_{0}})$, $g_{-}$ is incomplete.

\subsubsection{Case \texorpdfstring{$\sp = 0$}{sigma = 0} and  \texorpdfstring{$\rho \neq 0$}{rho = 0}}\label{cigarsection}
Suppose $\sp = 0$ and $\rho = \tau^{2}/4 > 0$. Equation \eqref{weq} reduces to $4w_{s}^{2} = \tau^{2}(w - \tau/4)^{2}$. The solution is $w(t, s) = 4\tau^{-1} + a(t) e^{\pm \tau s/2}$ for some function $a(t)$. Replacing $s$ by $-s$ corresponds to replacing $Y$ by $-Y$; since this can be done a posteriori, it suffices to consider $w(t, s) = 4\tau^{-1} + a(t) e^{-\tau s/2}$. Since $w$ solves \eqref{wt}, there must hold $a_{t} = \tau t$, so $a = Ae^{\tau t}$ for some $A\in \rea$. The trivial solution $w = 4/\tau$ corresponds to rescaling a flat metric, so $A$ can be supposed to be nonzero. In this case, by replacing $s$ by a translate, $A$ can be rescaled as desired, so $w$ can be supposed to have the form $W = \tau^{-1}(4 + \ep e^{\tau(t - s/2)})$, where $\ep = \pm 1$. There result the Ricci flows
\begin{align}\label{uspzero}
h_{\ep}(t) = \frac{\tau(dr^{2} + ds^{2})}{4\left(1 + \ep e^{\tau(t  -  s/2)}\right)} = \frac{\tau|z|^{\tau/2 - 2}|dz|^{2}}{4(|z|^{\tau/2} +\ep e^{\tau t})} = \frac{\tau|dz|^{2}}{4|z|^{2}(1 + \ep e^{\tau t }|z|^{-\tau/2})}.
\end{align}
Via time dependent translations in $s$, the metrics $\hhpm$ are diffeomorphic images of the steady gradient Ricci solitons $g_{\pm} = h_{\pm}(0)$. 

It remains to analyze the dependence on $\ep$ and $\tau$, and the largest domains of definition of the resulting metrics. If $\ep = 1$ then it must be that $\tau > 0$, in which case $h_{+}$ is defined at least in the punctured plane. If $\ep = -1$ then $h_{-}$ is defined at least on $|z| > e^{2t}$. To analyze the behavior as $s \to \infty$, take $\tilde{z} = -z^{-1}$; in this coordinate the metric $h_{\ep}(t)$ has the form
\begin{align}\label{cigar}
\frac{\tau|d\tilde{z}|^{2}}{4|\tilde{z}|^{2}(1 + \ep e^{\tau t}|\tilde{z}|^{\tau/2})}.
\end{align}
Hence, in the case $\ep = 1$ and $\tau > 0$ and the case $\ep = -1$ and $\tau > 0$ the metric $\hhp$ and the metric $\hhm$, respectively, are defined on $|z| > e^{2t}$ and have cusps at infinity. On the other hand, in the case $\ep = -1$ and $\tau < 0$, the metric $\hhm$ is defined on $|z| > e^{2t}$ and has at the point at infinity a cone point of angle $-\pi\tau/2$. In particular, when $\tau = -4$, the metric $\hhm$ extends smoothly at the point at infinity. In the coordinate $\tilde{z}$, the metric $h_{-}(0)$ has the form $(1 - |\tilde{z}|^{2})^{-1}|d\tilde{z}|^{2}$. Writing $\tilde{z}= \sin(q)e^{\j \theta}$, gives $h_{-}(0) = dq^{2} + \tan^{2}(q)d\theta^{2}$, from which it is apparent that the metric $h_{-}(0)$ is the incomplete metric called the \textit{exploding soliton} in section $1.3.3$ of \cite{ricciflowii}. For $\tau \neq -4$, the metrics $\hhm$ are exploding solitons with conical singularities at the origin.

When $\ep = 1$, it follows from the second expression of \eqref{uspzero} that $\hhp$ has at $z = 0$ (when $s \to -\infty$) a conical singularity with angle $\pi\tau/2$. In particular, in the case $\tau = 4$ the metric $\hhp$ extends smoothly across $z = 0$; the resulting metric is the well known \textit{cigar soliton} of \cite{Hamilton-riccisurfaces}. For other values of $\tau$ these metrics are cigars with conical singularities at the tips.  

In all cases,
\begin{align}
\begin{split}
\sR_{\hh}  = -\fd_{\hh} &= \frac{\tau\ep e^{\tau t }}{e^{\tau s/2} + \ep  e^{\tau t }} = \frac{\tau}{1 + \ep e^{-\tau (t - s/2)}}= \frac{\tau\ep e^{\tau t }}{|z|^{\tau/2} + \ep  e^{\tau t}}.
\end{split}
\end{align}
If $\ep = 1$ then $\tau \geq \sR_{h} > 0$ with the limiting values $\tau$ and $0$ approached as $s \to -\infty$ and $s \to \infty$, respectively. In particular in this case $\sR_{h}$ is positive and bounded. If $\ep = -1$ then $\sR_{h}$ is negative and unbounded from below, approaching $0$ as $s \to \infty$ and approaching $-\infty$ as $s \to 2t$. Note also that $\sR_{h}$ extends smoothly at the cone point in both the $\ep = 1$ and the $\tau < 0$ cases.

Consider the case $\ep = 1$ and write $\hh = \hhp$. Consider the metric $\tilde{h}(\tilde{t})$, defined as in \eqref{uspzero}, but with respect to coordinates $(\tilde{r}, \tilde{s})$ in place of $(r, s)$, with $\tilde{\tau}$ in place of $\tau$, and with time parameter $\tilde{t}$. Let $Y = \pr_{s}$ and $\tilde{Y} = \pr_{\tilde{s}}$. Suppose $\tilde{\tau} = 1$ and define a diffeomorphism by $(\tilde{r}, \tilde{s}) = \phi(r, s) = (\tau r, \tau s)$. Set $\tilde{t} = \tau t$. It is easily checked that $\phi^{\ast}(\hh) = \tau^{-1}\tilde{h}(\tilde{t})$ and $\phi^{\ast}(Y) = \tau \tilde{Y}$. Hence the pulled back pair $\phi^{\ast}(\hh, Y) = (\phi^{\ast}\hh, \phi^{\ast}Y) = (\tau^{-1}\tilde{h}(\tilde{t}), \tau \tilde{Y})$ is scaling equivalent to the pair $(\tilde{h}(\tilde{t}), \tilde{Y})$. This shows that the solutions of \eqref{vlike} given by $\hhp$ for different values of $\tau$ are scaling equivalent when viewed as solutions on the punctured disk. On the other hand, if these solutions are regarded as metrics on the disk with conical singularities at the origin, then they are not equivalent, because when viewed as a map on the punctured plane $\phi$ is not a diffeomorphism and does not in general extend smoothly through the puncture.

By Theorem $26.3$ of \cite{Hamilton-formation}, a complete Ricci soliton on a surface having bounded curvature assuming somewhere its maximum is diffeomorphic to the cigar soliton. Hamilton's characterization of the cigar solition has been improved by P. Daskalopoulos and N. Sesum who in \cite{Daskalopoulos-Sesum} proved that a complete ancient Ricci flow on a surface must be a cigar soliton if it has bounded positive curvature and bounded width (in a sense defined in \cite{Daskalopoulos-Sesum}). Hamilton's way of constructing the cigar soliton in \cite{Hamilton-riccisurfaces} also yields examples on orbifolds, as was developed by L.-F. Wu in \cite{Wu-orbifolds}.
After the first version of the present article was posted there appeared \cite{Bernstein-Mettler} and \cite{Ramos} addressing the classification of gradient solitons on surfaces. Most of the metrics described above appear in some equivalent form in at least one of \cite{Bernstein-Mettler} or \cite{Ramos}. The article \cite{Ramos} shows that the a complete steady gradient Ricci soliton on a surface with curvature bounded from below is either a flat surface (possible with cone points) or one of the (possibly conical) cigar solitons \eqref{cigar}.

\subsection{Cases with \texorpdfstring{$\sp \neq 0$}{sigma = 0} and \texorpdfstring{$\rho = 0$}{rho = 0}}
When $\sp \neq 0$ it is convenient to separate the cases $\rho = 0$ and $\rho \neq 0$.
Suppose $\sp \neq 0$ and $\rho = 0$, so that $\tau^{2} = \sp$. Since $\sp \neq 0$ this implies that $\sp > 0$. After an appropriate shift in $t$ it can be supposed that $\sp(t) = t^{-2}$ and $\tau(t) = -t^{-1}$. By \eqref{wss}, $q = -2^{-1}\tau s^{2} + as + b$ for some real functions $a(t)$ and $b(t)$. From \eqref{weq} it follows that $4\varep = a^{2} + 2\tau b$. Hence $w =  -2^{-1}\tau s^{2} + as + (4\varep - a^{2})/(2\tau) =  -2^{-1}\tau(s - a/\tau)^{2} + 2\varep/\tau$. That $w$ solve \eqref{wt} forces $a_{t} = \tau a$, which has the solution $\al \tau$ for some $\al \in \rea$. Hence $w =  -2^{-1}\tau((s - \al)^{2} - 4\varep \tau^{-2})$. Consequently, replacing $s$ its (time independent) translate $s + \al$, it may be supposed that $a = 0$, so that $b = 2\varep \tau^{-1}$ and $w = -2^{-1}\tau(s^{2} - 4\varep\tau^{-2})$. 

\subsubsection{Case \texorpdfstring{$\sp > 0$}{sigma > 0}, \texorpdfstring{$\rho = 0$}{rho = 0}, \texorpdfstring{$\la = \j$}{lambda = i}}
It follows from \eqref{sitau} that if $\varep = -1$ and $\tau > 0$ then $Y$ is identically $0$, so this case can be excluded. Hence if $\varep = -1$ it can be supposed $\tau < 0$ so that $\tau = -t^{-1}$ and $t > 0$. In this case $w =  -2^{-1}\tau(s^{2} + 4\tau^{-2}) = 2^{-1}t^{-1}(s^{2} + 4t^{2})$ has no real roots, so is positive for all $s \in \rea$. This yields the metrics
\begin{align}\label{rhozeronoroots}
\hh = \tfrac{2t(dr^{2} + ds^{2})}{s^{2} + 4t^{2}} = \tfrac{2t|dz|^{2}}{|z|^{2}((\log|z|)^{2} + 4t^{2})},
\end{align}
defined for $t > 0$ on the punctured plane. From \eqref{rhozeronoroots} it is apparent that $\hh$ blows up as $|z| \to 0$ or $|z| \to \infty$, so that $\hh$ extends to no larger domain. Since $\hh$ is bounded from below by the complete constant curvature metric $t s^{-2}(dr^{2} + ds^{2})$ outside the compact annulus $|s| \leq 2t$, the metric $\hh$ is complete for each $t > 0$. The metrics \eqref{rhozeronoroots} constitute a complete immortal Ricci flow on the punctured plane. The curvature of \eqref{rhozeronoroots} satisfies   
\begin{align}
-t^{-1} < \sR_{\hh} = \tfrac{1}{t}\tfrac{4t^{2} - s^{2}}{4t^{2} + s^{2}} \leq t^{-1},
\end{align}
with equality on the right-hand side along the unit circle $s = 0$. The curvature is positive when $|s| < 2t$ and negative when $|s| > 2t$. As $t\to \infty$ the metrics $\hh$ converge pointwise to a flat metric on the cylinder, while as $t\to 0$, the homothetic metrics $\kk = \sqrt{\sp(t)}\hh$ converge pointwise to the complete scalar curvature $-1$ metric on the punctured plane. Each $\hh$ has finite total absolute curvature, $\int_{\com \setminus \{0\}}|\sR_{h}|\om_{h} = 4\pi t^{-1}$, while its total curvature $\int_{\com \setminus \{0\}} \sR_{h} \om_{h}$ is $0$.

\subsubsection{Case \texorpdfstring{$\sp > 0$}{sigma > 0}, \texorpdfstring{$\rho = 0$}{rho = 0}, \texorpdfstring{$\la = 1$}{lambda = 1}}
In the $\varep = 1$ case, $\tau = - t^{-1}$ and so $w = 2^{-1}t^{-1}(s^{2} - 4t^{2})$. For $t < 0$, $\tau(t) = -t^{-1}$ is positive, and $w$ is positive on the annulus $2t = -2\tau^{-1} < s < 2\tau^{-1} = -2t$. The metrics 
\begin{align}\label{vep1}
\hh = \tfrac{-2t(dr^{2} + ds^{2})}{4t^{2} - s^{2}} = \tfrac{-2t|dz|^{2}}{|z|^{2}(4t^{2} - (\log|z|)^{2})},
\end{align}
are an ancient Ricci flow defined for $t < 0$ on the annulus $s <2|t|$. The curvature of \eqref{vep1} is
\begin{align}\label{cvep1}
\sR_{\hh} = \tfrac{1}{t}\tfrac{4t^{2} + s^{2}}{4t^{2} - s^{2}}  \leq \min\{\tau, -\sqrt{\sp}\} = -|t|^{-1} < 0
\end{align}
with equality on the right-hand side when $s = 0$. When $s \to \pm 2t$, $\sR_{h} \to -\infty$, so that $\sR_{h}$ is strictly negative and unbounded from below. In this case $\hh$ is not complete. An integral curve of $-|Y|^{-1}_{\hh}JY = u^{-1/2}\pr_{s}$ is a geodesic, and such a curve can be parameterized as $(r(x), s(x)) = (0, 2|t|\sin(x/\sqrt{2|t|}))$, which evidently exists only for $x^{2} < -\pi^{2}t/2$.

Finally, suppose $\varep = 1$ and $\tau < 0$. The metrics \eqref{vep1} are defined for all $t > 0$ on $|s| > 2t$. From \eqref{vep1} it is apparent that $\hh$ blows up when $s \to \pm\infty$, so that the maximal connected domains of definition of $\hh$ are the half-infinite cylinders constituting $|s| > 2t$, each of which is biholomorphic to the punctured disk. These components are interchanged by the map $s \to -s$, which maps the solution $(\hh, Y)$ on one component to the solution $(\hh, -Y)$ on the other component. Its curvature is as in \eqref{cvep1}. An integral curve of $|Y|^{-1}_{\hh}JY = -u^{-1/2}\pr_{s}$ is a geodesic. As such a curve can be parameterized as $r(x) = 0$ and $s(x) = -2t\cosh((c - x)/\sqrt{2t})$ for $x \in (-\infty, c)$, where $s(0) = s_{0} < -2t$ and $\cosh (c/\sqrt{2t}) = -s_{0}/(2t)$, $\hh$ is not complete.
Thus $\hh$ is an immortal solution to the Ricci flow with negative curvature unbounded from below. As $t \to 0$ the rescaled metric $\kk = \sqrt{\sp}(t)\hh$ tends pointwise to the constant scalar curvature $-1$ metric $2s^{-2}(dr^{2} + ds^{2})$.

\subsection{Generalities related to the cases where \texorpdfstring{$\sp \neq 0$}{sigma not 0} and \texorpdfstring{$\rho \neq 0$}{rho not 0}}\label{rhonotzerosection}
Suppose $\sp \neq 0$ and $\rho \neq 0$. Then $\tau$ and $\sp$ have one of the forms \eqref{taut1}-\eqref{tautneg}. Then $w$ has the form
\begin{align}\label{wrhop0}
w = \rho^{-1}\tau + A\cosh \sqrt{\rho}s + B\sqrt{\sign(\rho)}\sinh \sqrt{\rho}s,
\end{align}
where $A(t)$ and $B(t)$ are real functions of $t$, $\sqrt{\sign(\rho)}$ means $1$ or $\j$ as $\rho$ is positive or negative, and $\sqrt{\rho}$ means $\sqrt{\sign(\rho)}\sqrt{|\rho|}$. Substituting \eqref{wrhop} into \eqref{weq} yields $\sp = \rho^{2}(A^{2} - \sign(\rho)B^{2})$. In particular, the assumption $\sp \neq 0$ precludes the simultaneous vanishing of both $A$ and $B$. That $w$ satisfy \eqref{wt} forces $A_{t} = \tau A$ and $B_{t} = \tau B$, from which it follows that there are constants $a, b \in \rea$ such that $A = a\rho^{-1}|\sp|^{1/2}$ and $B = b\rho^{-1}|\sp|^{1/2}$, so that
\begin{align}\label{wrhop}
w = \rho^{-1}\left(\tau + |\sp|^{1/2}\left(a\cosh \sqrt{\rho}s + b\sqrt{\sign(\rho)}\sinh \sqrt{\rho}s \right)\right),
\end{align}
and $a^{2} - \sign(\rho)b^{2} = \sign(\sp)$. If $\sp > 0$ and $\rho > 0$ then $a^{2} - b^{2} = 1$ so there are a unique real number $q$ and a sign $\ep = \pm 1$ such that $\ep\cosh q = a$ and $\ep\sinh q = b$. Hence $w = \rho^{-1}\left(\tau + \ep|\sp|^{1/2}\cosh (\sqrt{\rho}s + q )\right)$. If $\sp > 0$ and $\rho < 0$ then $a^{2} + b^{2} = 1$ so there is a unique $q \in [0, 2\pi)$ such that $-\cos q = -\cosh\sqrt{\sign(\rho)}q = a$ and $-\sin q = -\sqrt{\sign(\rho)}\sinh \sqrt{\sign(\rho)}q = b$ (the choice of sign is arbitrary). Hence $w = \rho^{-1}\left(\tau - |\sp|^{1/2}\cosh (\sqrt{\rho}s + \sqrt{\sign(\rho)}q )\right)$. If $\sp < 0$ then, by Lemma \ref{spneglemma}, $\rho > 0$, and so $b^{2} - a^{2}= 1$ and there are a unique real number $q$ and a sign $\ep = \pm 1$ such that $\ep\cosh q = b$ and $\ep\sinh q = a$. Hence $w = \rho^{-1}\left(\tau +\ep |\sp|^{1/2}\sinh (\sqrt{\rho}s + q )\right)$. In all cases, after a translation in $s$ it can be supposed that $q = 0$. In the case $\sp < 0$, the sign $\ep$ can be eliminated by replacing $s$ by $-s$; although this corresponds to replacing $Y$ by $-Y$, no generality is lost because the discarded solution can be recovered a posteriori. There result for $w$ the following forms.
\begin{align}\label{wrho}
w = \begin{cases}
 \rho^{-1}\left(\tau + \ep \sqrt{\sp}\cosh \sqrt{\rho}s \right) & \,\,\text{if}\,\, \sp > 0 \,\,\text{and}\,\, \rho > 0,\\
 \rho^{-1}\left(\tau - \sqrt{\sp}\cos \sqrt{|\rho|}s \right) & \,\,\text{if}\,\, \sp > 0 \,\,\text{and}\,\, \rho < 0,\\
 \rho^{-1}\left(\tau + |\sp|^{1/2}\sinh \sqrt{\rho}s \right)& \,\,\text{if}\,\, \sp < 0.
\end{cases}
\end{align}
The $\ep = -1$ case of the first expression in \eqref{wrho} and the second expression in \eqref{wrho} are really the same; if $\rho$ is negative, then the first expression gives the second via the identity $\cosh \j x = \cos x$. The expressions for $\tau$ and $\sp$ can be taken as in \eqref{taut1}-\eqref{tautneg}, depending on the value of $\varep$. The various possibilities are analyzed in more detail in the sections to follow.

Note that the metrics \eqref{uspzero} in the $\sp = 0$ case arise from the $A = \pm B$ case of \eqref{wrhop0}. 

The ansatz $u = 2\la^{2}(a(t) + b(t)\cosh{2\la s})^{-1}$ for solutions of $u_{t} = (\log u)_{ss}$, in which $\la$ is either $1$ or $\j$ and $a(t)$ and $b(t)$ are real functions defined on some connected open subset of $\rea$, was used by Fateev-Onofri-Zamolodchikov in \cite{Fateev-Onofri-Zamolodchikov} to find solutions to the Ricci flow (regarded in \cite{Fateev-Onofri-Zamolodchikov} as the one-loop approximation to the renormalization group flow). In \cite{Fateev-Onofri-Zamolodchikov} the extra generality of the parameter $\la$ was not needed because solutions with $\la = \j$ were excluded by physical considerations (in this regard see \cite{Bakas}). 

\subsection{Cases with \texorpdfstring{$\sp > 0$}{sp > 0}}
Combining \eqref{sptau} and \eqref{wrho} shows that when $\sp > 0$ case the Ricci flow $\hh$ such that $(\hh, Y)$ solves \eqref{vlike} is given by
\begin{align}\label{usppos}
\begin{split}
\hh &= \frac{-\sqrt{\rho}\sinh (2\la \sqrt{\rho}t)(dr^{2} + ds^{2})}{2\la\left(\cosh 2\la \sqrt{\rho}t  + \ep \cosh \sqrt{\rho}s \right)} = \frac{-\ep\sqrt{\rho}\sinh (2\la \sqrt{\rho} t)|z|^{\sqrt{\rho}-2}|dz|^{2}}{\la \left(|z|^{2\sqrt{\rho}} + 2\ep \cosh (2\la \sqrt{\rho}t)|z|^{\sqrt{\rho}}  +  1\right)}.
\end{split}
\end{align}
The expression \eqref{usppos} encodes various qualitatively different metrics, depending on the values of the various parameters. In particular it is convenient to separate the cases $\rho > 0$ and $\rho < 0$. These cases are detailed separately in sections \ref{sppp} and \ref{sppn}.

When the coordinate $z$ is replaced by $\tilde{z} = -z^{-1}$ the form of the last expression in \eqref{usppos} is unchanged. It follows that, in the cases where it makes sense, the metric $\hh$ has at the origin $z = 0$ or at the point at infinity a conical singularity with angle $\pi\sqrt{\rho}$.

The scalar curvature of \eqref{usppos} is
\begin{align}\label{uspposcurv}
\begin{split}
\sR_{\hh} &= \frac{\sp + \ep \tau \sqrt{\sp}\cosh \sqrt{\rho}s}{\tau + \ep\sqrt{\sp}\cosh \sqrt{\rho}s}
 = \frac{-2\ep \la \sqrt{\rho}}{\sinh 2\la\sqrt{\rho}t}\frac{\cosh (2\la\sqrt{\rho})t\cosh \sqrt{\rho}s + \ep}{\cosh 2\la \sqrt{\rho}t + \ep\cosh \sqrt{\rho}s}.
\end{split}
\end{align}
Similarly,
\begin{align}\label{uspposf}
\fd_{\hh} = \frac{2\ep\sqrt{\sp}\sqrt{\rho} \sinh \sqrt{\rho}s}{\tau + \ep\sqrt{\sp}\cosh\sqrt{\rho}s} = \frac{2\ep\sqrt{\rho}\sinh\sqrt{\rho}s}{\cosh 2\la \sqrt{\rho}t + \ep \cosh \sqrt{\rho}s}.
\end{align}


\subsection{Cases with \texorpdfstring{$\sp > 0$}{sp > 0} and \texorpdfstring{$\rho > 0$}{rho > 0}}\label{sppp}

\subsubsection{Case \texorpdfstring{$\sp > 0$}{sigma > 0}, \texorpdfstring{$\rho > 0$}{rho > 0}, \texorpdfstring{$\la = 1$}{lambda = 1}, \texorpdfstring{$\ep = 1$}{epsilon = 1}}
The metrics
\begin{align}\label{sausage}
\hh = \frac{-\sqrt{\rho}\sinh(2\sqrt{\rho}t)(dr^{2} + ds^{2})}{2\left(\cosh 2\sqrt{\rho}t + \cosh \sqrt{\rho}s\right)} = \frac{-\sqrt{\rho}\sinh (2\sqrt{\rho}t)|z|^{\sqrt{\rho}-2}|dz|^{2}}{|z|^{2\sqrt{\rho}} + 2\cosh (2\sqrt{\rho}t)|z|^{\sqrt{\rho}}  +  1}
\end{align}
constitute an ancient Ricci flow, being defined for $t \in (-\infty, 0)$. They extend to all of $\sphere$ with conical singularities of angle $\pi \sqrt{\rho}$ at the origin and the point at infinity, which are the zeros of $Y$. In the particular case $\rho = 4$, the metrics $\hh$ extend smoothly to all of $\sphere$. These metrics are often called the \textit{King-Rosenau} metrics because the corresponding solutions of the logarithmic diffusion equation were found by P. Rosenau in \cite{Rosenau} and J.~R. King in \cite{King-exactpolynomial}. These metrics were found independently by V. Fateev, E. Onofri, and A.~B. Zamolodchikov in \cite{Fateev-Onofri-Zamolodchikov}, who used the more descriptive appellation \textit{sausage metric} used here. The main theorem of \cite{Daskalopoulos-Hamilton-Sesum} shows that an ancient solution to the Ricci flow on a compact surface is diffeomorphsim equivalent to either a contracting sphere or the sausage metrics. 

By \eqref{usppos}, the curvature $\sR_{h}$ is
\begin{align}\label{sausagecurvature}
\begin{split}
\sR_{\hh}  &= \frac{-2\sqrt{\rho}}{\sinh 2\sqrt{\rho}t}\frac{\cosh (2\sqrt{\rho}t)\cosh \sqrt{\rho}s + 1}{\cosh 2 \sqrt{\rho}t + \cosh \sqrt{\rho}s}.
\end{split}
\end{align}
It extends smoothly to all of $\sphere$ and, by \eqref{ub1} of Lemma \ref{squaresconstantlemma}, it is strictly positive, satisfying
\begin{align}\label{sausbound}
0 <  -2\sqrt{\rho}\csch(2\sqrt{\rho} t) = \sqrt{\sp(t)} \leq \sR_{\hh} \leq \tau(t) =  -2\sqrt{\rho}\coth(2\sqrt{\rho}t),
\end{align}
with equality on the left-hand side exactly along the equatorial geodesic $s = 0$ where $\fd_{h}$ vanishes, and with equality on the right-hand side exactly where $Y$ vanishes, at the cone points.

By \eqref{sausbound} the homothetic metrics $\kk = \tau(t)\hh = -2\sqrt{\rho}\coth(2\sqrt{\rho}t)\hh$
have curvature satisfying $0 < \sech 2\sqrt{\rho}t \leq \sR_{\kk} \leq 2$.  As $t \to 0$ the metrics $\kk$ converge pointwise to the constant curvature $1$ conical metric \eqref{rhospheremetric} on $\sphere$ having two cone points of angle $\pi\sqrt{\rho}$. As $t \to -\infty$, the metrics $\kk$ converge pointwise to the flat metric $\rho(dr^{2} + ds^{2})$ on the punctured plane. 

\subsubsection{Case \texorpdfstring{$\sp > 0$}{sigma > 0}, $\rho > 0$, $\la = \j$, $\ep = \pm 1$} 
As will be explained next, if $\sp > 0$, $\rho > 0$, and $\la = \j$, since $\cos(\pi - x) = -\cos x$ and $\sin(\pi - x) = \sin x$, the two cases obtained by taking $\ep = 1$ or $\ep = -1$ are equivalent up to an orientation-reversing unimodular transformation of $t$, and so it suffices to consider the case $\ep = 1$. In the case $\ep = 1$, the metrics
\begin{align}\label{ewmetric}
\hh =  \frac{-\sqrt{\rho}\sin(2\sqrt{\rho}t)(dr^{2} + ds^{2})}{2\left(\cos 2\sqrt{\rho}t + \cosh \sqrt{\rho}s\right)}  = \frac{-\sqrt{\rho}\sin (2 \sqrt{\rho}t)|z|^{\sqrt{\rho}-2}|dz|^{2}}{|z|^{2\sqrt{\rho}} + 2 \cos (2 \sqrt{\rho}t)|z|^{\sqrt{\rho}}  +  1}
\end{align}
are defined for $\sqrt{\rho}t \in (-\pi/2, 0)$. They extend to all of $\sphere$ with conical singularities of angle $\pi \sqrt{\rho}$ at the origin and the point at infinity, which are the zeros of $Y$. When $\rho = 4$ the metrics $\hh$ extend smoothly to the entire two sphere. A straightforward calculation shows that the metric $h^{-}(t)$ obtained from \eqref{usppos} with $\sp > 0$, $\rho > 0$, $\la = \j$, and $\ep = - 1$ is related to $\hh$ by $h^{-}(t) = h(-t - 2^{-1}\rho^{-1/2}\pi)$. Notice that as a consequence there must hold $\sR_{h^{-}(t)} = -\sR_{h(-t - 2^{-1}\rho^{-1/2}\pi)}$, which can also be verified directly using \eqref{uspposcurv}. Although henceforth there is considered only the metric $\hh$ of the case $\ep = 1$, it is a special property of the Ricci flow $\hh$ that it remains a Ricci flow under (shifted) time reversal.

By \eqref{uspposcurv}, and \eqref{ub1b} of Lemma \ref{squaresconstantlemma}, the curvature of \eqref{ewmetric} satisfies
\begin{align}\label{ewcurvature}
\begin{split}
-2\sqrt{\rho}\cot 2\sqrt{\rho}t = \tau(t) \leq \sR_{\hh} & = \frac{-2\sqrt{\rho}}{\sin 2\sqrt{\rho}t}\frac{\cos (2\sqrt{\rho}t)\cosh \sqrt{\rho}s + 1}{\cos 2\sqrt{\rho}t + \cosh \sqrt{\rho}s}\leq \sqrt{\sp(t)} = \frac{-2\sqrt{\rho}}{\sin 2\sqrt{\rho}t}.
\end{split}
\end{align}
Although $\sR_{\hh} > 0$ if $\sqrt{\rho}t \in (-\pi/4, 0)$, it is somewhere negative for $\sqrt{\rho}t \in ( -\pi/2, - \pi/4)$. It attains its maximum along the equatorial geodesic circle $s = 0$ where $\fd_{\hh}$ vanishes, and tends to its minimum at the cone points, when $s \to \pm \infty$. 
From \eqref{ewcurvature} it is apparent that $\sR_{h}$ is positive exactly where $\cosh \sqrt{\rho}s < -\sec 2\sqrt{\rho}t$. Solving this inequality for $e^{\sqrt{\rho}s}$ yields the equivalent inequalities
\begin{align}\label{ew1}
 -\sec 2\sqrt{\rho}t + \tan 2\sqrt{\rho}t  < e^{\sqrt{\rho}s} <  -\sec2\sqrt{\rho}t - \tan 2\sqrt{\rho}t .
\end{align}
Simplifying \eqref{ew1} using the identity $\tan((a + b)/2) = (\sin a + \sin b )/(\cos a + \cos b)$ with $b = \pi/2$ and $a = -2\sqrt{\rho}t$ yields that, for $\sqrt{\rho}t \in ( -\pi/2, - \pi/4)$ the curvature is positive on the equatorial band $|s| < \rho^{-1/2}\log \tan|\sqrt{\rho}t + \pi/4|$.  

The homothetic metrics $\kk = \sqrt{\sp}(t)\hh = -2\sqrt{\rho}\csc(2\sqrt{\rho}t)\hh$
have curvature satisfying $-2 \leq \cos 2\sqrt{\rho}t \leq \sR_{\kk} \leq 2$. As $t\to 0$ the metrics $\kk$ tend to the constant curvature $1$ conical metric \eqref{rhospheremetric} on $\sphere$, 
while when $t \to - 2^{-1}\rho^{-1/2}\pi$ the metrics $\kk$ converge pointwise to the constant curvature $-1$ metric \eqref{rhohyperbolicmetric} defined on the disks complementary to the equator $s = 0$ and having cone points of angle $\pi\sqrt{\rho}$ at the centers of these disks. When $\rho = 4$ the metrics $\kk$ interpolate between the spherical metric and the hyperbolic metric.

\subsubsection{Case \texorpdfstring{$\sp > 0$}{sigma > 0}, $\rho > 0$, $\la = 1$, $\ep = -1$}
For the metrics
\begin{align}
\hh =  
 \frac{\sqrt{\rho}\sinh(2\sqrt{\rho}t)(dr^{2} + ds^{2})}{2\left(\cosh \sqrt{\rho}s - \cosh 2\sqrt{\rho}t \right)}   = \frac{\sqrt{\rho}\sinh (2\sqrt{\rho}t)|z|^{\sqrt{\rho}-2}|dz|^{2}}{|z|^{2\sqrt{\rho}} - 2 \cosh (2\sqrt{\rho}t)|z|^{\sqrt{\rho}}  +  1}
\end{align}
there are two cases, distinguished by the sign of $\tau(t) = -2\sqrt{\rho}\coth 2\sqrt{\rho}t $. When $\tau(t)$ is positive, $\hh$ is defined on the annulus $- 2t > |s|$ for $t < 0$. When $\tau$ is negative, $\hh$ is defined on $2t < |s|$ for $t > 0$, and has cone points with angle $\pi\sqrt{\rho}$ when $s \to \pm \infty$, so is the disjoint union of two infinite components each with a form something like a capped funnel (they are smooth when $\rho = 4$). In the $\tau < 0$ case the map sending $s$ to $-s$ (equivalently $z$ to $-z^{-1}$) extends to an involution interchanging the two connected components. 

In both cases the curvature
\begin{align}\label{iicurvature}
\begin{split}
\sR_{\hh} & = \frac{2\sqrt{\rho}}{\sinh 2\sqrt{\rho}t }\frac{\cosh (2\sqrt{\rho}t)\cosh \sqrt{\rho}s - 1}{\cosh 2\sqrt{\rho}t - \cosh \sqrt{\rho}s},
\end{split}
\end{align}
is strictly negative on the domain of definition of $\hh$. It is not bounded from below for it tends to $-\infty$ as $|s| \to 2|t|$. By \eqref{ub1} of Lemma \ref{squaresconstantlemma}, in the $\tau > 0$ case, $\sR_{h} \leq -\sqrt{\sp} = 2\sqrt{\rho}\csch 2\sqrt{\rho}t$, with equality along the equatorial geodesic $s =0$, where $\fd_{\hh} = 0$.
Since $\sp = \tau^{2} + 4\rho \geq \tau^{2}$, when $\tau < 0$ there holds $\tau < -\sqrt{\sp}$. Hence, by \eqref{ub1} of Lemma \ref{squaresconstantlemma}, in the $\tau < 0$ case, $\sR_{h} \leq \tau = -2\sqrt{\rho}\coth 2\sqrt{\rho}t$, with equality when $s \to \pm \infty$, that is, at the cone points.

When $\tau < 0$ the homothetic metrics $\kk = \sqrt{\sp(t)}\hh = 2\sqrt{\rho}\csch(2\sqrt{\rho}t)\hh$
converge pointwise, as $t$ tends to $0$ from above, to the constant curvature $-1$ singular metric \eqref{rhohyperbolicmetric} on the disjoint union of two disks. 
When $\tau > 0$ the homothetic metrics $\tkk = \tau(t)\hh = -2\sqrt{\rho}\coth(2\sqrt{\rho}t)\hh$
converge pointwise, as $t$ tends to $0$ from below, to the flat cylindrical metric $\rho(dr^{2} + ds^{2})$.

\subsection{Cases with \texorpdfstring{$\sp > 0$}{sigma > 0} and \texorpdfstring{$\rho < 0$}{rho < 0}}\label{sppn}
As noted in the derivation of \eqref{wrho}, in the case $\sp > 0$,  $\rho < 0$ there is no need for the sign parameter $\ep$ that appeared in the case $\sp > 0$, $\rho > 0$. 
\subsubsection{Case \texorpdfstring{$\sp > 0$}{sigma > 0}, \texorpdfstring{$\rho < 0$}{rho < 0}, \texorpdfstring{$\la = 1$}{lambda = 1}}
The metric
\begin{align}
\hh =  2^{-1}\sqrt{|\rho|}\sin(2\sqrt{|\rho|}t)\left(\cos 2\sqrt{|\rho|}t - \cos \sqrt{|\rho|}s\right)^{-1}(dr^{2} + ds^{2})
\end{align}
can be taken to be defined for $t \in (0, 2^{-1}|\rho|^{-1/2}\pi)$ in the bounded open cylinder $s \in (2t,  2|\rho|^{-1/2}\pi - 2t)$, or for $t \in (2^{-1}|\rho|^{-1/2}\pi , |\rho|^{-1/2}\pi)$ in the bounded open cylinder $s \in (2|\rho|^{-1/2}\pi- 2t,  2t)$. However, a straightforward calculation shows $h(|\rho|^{-1/2}\pi- s, -t + 2^{-1}|\rho|^{-1/2}\pi) = h(s, t)$, so that, modulo an orientation-reversing unimodular transformation of the time parameter, these flows are equivalent modulo a reflection in $s$. For this reason, $\hh$ will be considered for $t \in (0,  2^{-1}|\rho|^{-1/2}\pi)$.

Since $\varep\rho < 0$, $-\sqrt{\sp} \leq \tau$, and so, by \eqref{ub1} of Lemma \ref{squaresconstantlemma}, the curvature is negative, satisfying
\begin{align}\label{iiicurvature}
\begin{split}
\sR_{\hh} & = \frac{2\sqrt{|\rho|}}{\sin 2\sqrt{|\rho|}t}\frac{\cos (2\sqrt{|\rho|}t)\cos \sqrt{|\rho|}s - 1}{\cos 2\sqrt{|\rho|}t - \cos \sqrt{|\rho|}s} \leq -\sqrt{\sp} = -2\sqrt{|\rho|}\csc 2\sqrt{|\rho|}t < 0.
\end{split}
\end{align}
The curvature assumes its maximum value  $-\sqrt{\sp} = -2\sqrt{|\rho|}\csc 2\sqrt{|\rho|}t$ on the equatorial circle $s = |\rho|^{-1/2}\pi$, where $\fd_{h} = 0$. It is unbounded from below, blowing up as $s \to 2t$ or $s \to 2|\rho|^{-1/2}\pi - 2t$. 
As $t \to 0$, the homothetic metrics $\kk = \sqrt{\sp}\hh =  2\sqrt{|\rho|}\csc(2\sqrt{|\rho|}t)\hh$ tend pointwise to the constant curvature $-1$ metric $|\rho|(1 - \cos \sqrt{|\rho|}s)^{-1}(dr^{2} + ds^{2})$ on the infinite cylinder.

\subsubsection{Case \texorpdfstring{$\sp > 0$}{sigma > 0}, \texorpdfstring{$\rho < 0$}{rho < 0}, \texorpdfstring{$\la = \j$}{lambda = i}}
The metric
\begin{align}
\hh =  2^{-1}\sqrt{|\rho|}\sinh(2\sqrt{|\rho|}t)\left(\cosh 2\sqrt{|\rho|}t - \cos \sqrt{|\rho|}s\right)^{-1}(dr^{2} + ds^{2}).
\end{align}
is defined for all $t > 0$ and all $s \in \rea$, that is on the punctured plane. Because $\hh$ has period $2\pi |\rho|^{-1/2}$ in $s$, it descends to the torus $\{(r, s)\in [0, 2\pi) \times [0, 2\pi |\rho|^{-1/2})\}$. The curvature assumes both positive and negative values, and, by \eqref{ub1b} of Lemma \ref{squaresconstantlemma}, it satisfies
\begin{align}\label{toruscurvature}
\begin{split}
-\sqrt{\sp} \leq  \sR_{\hh} & = \frac{2\sqrt{|\rho|}}{\sinh 2\sqrt{|\rho|}t}\frac{\cosh (2\sqrt{|\rho|}t)\cos \sqrt{|\rho|}s - 1}{\cosh 2\sqrt{|\rho|}t - \cos \sqrt{|\rho|}s} \leq \sqrt{\sp} = 2\sqrt{|\rho|}\csch 2\sqrt{|\rho|}t,
\end{split}
\end{align}
The maximum and minimum are attained, respectively, along the geodesics $s = 0$ and $s = \pi |\rho|^{-1/2}$, where $\fd_{h}$ vanishes. This Ricci flow is interesting because it is immortal, the manifold is compact, and the $\hh$ have bounded curvature. As $t \to 0$, the homothetic metrics $\kk = \sqrt{\sp}\hh =  2\sqrt{|\rho|}\csch(2\sqrt{|\rho|}t)\hh$ on the infinite cylinder
tend pointwise to the constant curvature $-1$ metric $|\rho|(1 - \cos \sqrt{|\rho|}s)^{-1}(dr^{2} + ds^{2})$ on the infinite cylinder.

\subsection{Case \texorpdfstring{$\sp < 0$}{sp < 0}}
By Lemma \ref{spneglemma}, if $\sp < 0$ then $\rho > 0$, and so $w$ has the form as in the last expression of \eqref{wrhop}, with $\sp$ and $\tau$ as in \eqref{tautneg}. Explicitly, for $t \in \rea$, the metric $\hh$ is defined on the complement $s > 2t$ of the disk of radius $e^{t}$ by
\begin{align}\label{uspneg}
\begin{split}
\hh &= \frac{\sqrt{\rho}\cosh (2\sqrt{\rho}t)(dr^{2} + ds^{2})}{2\left( \sinh \sqrt{\rho}s - \sinh 2\sqrt{\rho}t\right)} = \frac{\sqrt{\rho}\cosh (2\sqrt{\rho} t)|z|^{\sqrt{\rho} - 2}|dz|^{2}}{\left(|z|^{2\sqrt{\rho}} - 2 \sinh (2\sqrt{\rho}t)|z|^{\sqrt{\rho}}  - 1 \right)}\\
& = \frac{\sqrt{\rho}\cosh (2\sqrt{\rho}t)|\tilde{z}|^{\sqrt{\rho} - 2}|d\tilde{z}|^{2}}{\left(1 - 2 \sinh (2\sqrt{\rho}t)|\tilde{z}|^{\sqrt{\rho}} -|\tilde{z}|^{2\sqrt{\rho}}\right)} = \frac{\sqrt{\rho}\cosh (2\sqrt{\rho}t)|\tilde{z}|^{\sqrt{\rho} - 2}|d\tilde{z}|^{2}}{\left(e^{2\sqrt{\rho}t}+|\tilde{z}|^{\sqrt{\rho}}\right)\left(e^{-2\sqrt{\rho}t}  -|\tilde{z}|^{\sqrt{\rho}}\right)},
\end{split}
\end{align}
where $\tilde{z} = z^{-1}$. From the last expression in \eqref{uspneg} it is apparent that the metric $\hh$ has a conical singularity at the point at infinity with angle $\pi \sqrt{\rho}$. In particular, in the case $\rho = 4$, the metric $\hh$ extends smoothly to the point at infinity. By \eqref{ub1} of Lemma \ref{squaresconstantlemma}, the scalar curvature of the metric $\hh$ of \eqref{uspneg} satisfies
\begin{align}
\begin{split}
\sR_{\hh} & = \frac{-2\sqrt{\rho}}{\cosh 2\sqrt{\rho}t}\frac{\sinh 2\sqrt{\rho}t \sinh \sqrt{\rho}s + 1 }{\sinh \sqrt{\rho}s - \sinh 2 \sqrt{\rho} t} \leq \tau(t) = -2\sqrt{\rho}\tanh 2\sqrt{\rho}t,
\end{split}
\end{align}
and as $s \to 2t$ the curvature tends to $-\infty$, while as $s \to \infty$, it tends to $\tau(t)$. There hold
\begin{align}
&\fd_{\hh} = \frac{2\sqrt{\rho}\cosh\sqrt{\rho}s}{ \sinh \sqrt{\rho}s - \sinh 2\sqrt{\rho}t},&
&\sR_{h} \pm \fd_{h} = \frac{2\sqrt{\rho}}{\cosh 2\sqrt{\rho} t}\frac{\left(\pm \cosh (\sqrt{\rho}(s \mp 2t)) - 1\right)}{ \sinh \sqrt{\rho}s - \sinh 2\sqrt{\rho}t}.
\end{align}
By Lemma \ref{squaresconstantlemma}, $\sR_{h} + \fd_{h}$ and $\sR_{h} - \fd_{h}$ must have definite (and opposite) signs. Here $\sR_{h} + \fd_{h}$
is positive and $\sR_{h} - \fd_{h}$ is negative. That here $\sR_{h} + \fd_{h}$ is positive results from the choice of sign in the construction of $\hh$ described just before \eqref{wrho}, which could be recast invariantly as demanding that $Y$ be such that the sign of $\sR_{h} - \fd_{h}$ be positive.

The Ricci flow $\hh$ is remarkable because it is eternal (exists for all time). This does not contradict the main theorem of \cite{Daskalopoulos-Sesum}, which classifies complete eternal Ricci flows on surfaces having bounded curvature and bounded width, because the curvature is unbounded from below.


\def\cprime{$'$} \def\cprime{$'$} \def\cprime{$'$}
  \def\polhk#1{\setbox0=\hbox{#1}{\ooalign{\hidewidth
  \lower1.5ex\hbox{`}\hidewidth\crcr\unhbox0}}} \def\cprime{$'$}
  \def\Dbar{\leavevmode\lower.6ex\hbox to 0pt{\hskip-.23ex \accent"16\hss}D}
  \def\cprime{$'$} \def\cprime{$'$} \def\cprime{$'$} \def\cprime{$'$}
  \def\cprime{$'$} \def\cprime{$'$} \def\cprime{$'$} \def\cprime{$'$}
  \def\cprime{$'$} \def\cprime{$'$} \def\cprime{$'$} \def\cprime{$'$}
  \def\dbar{\leavevmode\hbox to 0pt{\hskip.2ex \accent"16\hss}d}
  \def\cprime{$'$} \def\cprime{$'$} \def\cprime{$'$} \def\cprime{$'$}
  \def\cprime{$'$} \def\cprime{$'$} \def\cprime{$'$} \def\cprime{$'$}
  \def\cprime{$'$} \def\cprime{$'$} \def\cprime{$'$} \def\cprime{$'$}
  \def\cprime{$'$} \def\cprime{$'$} \def\cprime{$'$} \def\cprime{$'$}
  \def\cprime{$'$} \def\cprime{$'$} \def\cprime{$'$} \def\cprime{$'$}
  \def\cprime{$'$} \def\cprime{$'$} \def\cprime{$'$} \def\cprime{$'$}
  \def\cprime{$'$} \def\cprime{$'$} \def\cprime{$'$} \def\cprime{$'$}
  \def\cprime{$'$} \def\cprime{$'$} \def\cprime{$'$} \def\cprime{$'$}
  \def\cprime{$'$} \def\cprime{$'$} \def\cprime{$'$} \def\cprime{$'$}
\providecommand{\bysame}{\leavevmode\hbox to3em{\hrulefill}\thinspace}
\providecommand{\MR}{\relax\ifhmode\unskip\space\fi MR }
\providecommand{\MRhref}[2]{%
  \href{http://www.ams.org/mathscinet-getitem?mr=#1}{#2}
}
\providecommand{\href}[2]{#2}

\end{document}